\documentclass[12pt,a4paper,reqno]{amsart}

\usepackage{amssymb}
\usepackage{amscd}
\usepackage[pdftex,pdfpagelabels]{hyperref}
\usepackage{enumerate}
\usepackage{graphicx}
\usepackage{siunitx}
\usepackage{tikz-cd}
\usepackage{stix}
\usepackage{bm}
\usepackage{cleveref}

\numberwithin{equation}{section}

\usepackage{mathtools}%                  http://www.ctan.org/pkg/mathtools
\usepackage[tableposition=top]{caption}% http://www.ctan.org/pkg/caption
\usepackage{booktabs,dcolumn}%           http://www.ctan.org/pkg/dcolumn + http://www.ctan.org/pkg/booktabs

\DeclareFontFamily{OT1}{rsfs}{}
\DeclareFontShape{OT1}{rsfs}{n}{it}{<-> rsfs10}{}
\DeclareMathAlphabet{\mathscr}{OT1}{rsfs}{n}{it}

\theoremstyle{plain}

\newtheorem{theorem}{Theorem}[section]
\newtheorem{proposition}[theorem]{Proposition}
\newtheorem{lemma}[theorem]{Lemma}

\theoremstyle{definition}

\newtheorem{definition}[theorem]{Definition}
\newtheorem{remark}[theorem]{Remark}

\newtheorem{example}[theorem]{Example}

\renewcommand\P{\mathbf{P}}
\newcommand\E{\mathbf{E}}

\newcommand\R{\mathbb{R}}
\newcommand\Z{\mathbb{Z}}
\newcommand\N{\mathbb{N}}

\newcommand\Trap{{\operatorname{Trap}}}
\newcommand\eps{\varepsilon}

\renewcommand{\mod}{\bmod}

\newcommand{\hide}[1]{}

\begin{document}

\title[Gilbreath's conjecture for Cram\'er model]{Gilbreath's conjecture: a Cram\'er random model\\ and a deterministic analysis}

\author{Zachary Chase}
\address{Department of Mathematics, Kent State University}
\email{zachman99323@gmail.com}

\author{Zach Hunter}
\address{ETH Zurich, Department of Mathematics, R\"amistrasse 101, 8092 Zurich, Switzerland}
\email{zach.hunter@math.ethz.ch}

\author{Terence Tao}
\address{UCLA Department of Mathematics, Los Angeles, CA 90095-1555.}
\email{tao@math.ucla.edu}


\begin{abstract}  Gilbreath's conjecture asserts that if one starts with the sequence of primes and takes successive absolute  differences to create a triangular array, then the left diagonal of this array consists entirely of ones after the first row.  In this paper, we show that the analogue of this conjecture for a Cram\'er random model holds, in which the (normalized) prime gaps are replaced by independent random variables with geometric distributions of logarithmic size.  We also give some preliminary analysis of the associated continuous probabilistic model for this problem, as well as a deterministic ``inverse theorem'' that isolates the specific obstructions to Gilbreath's conjecture (assuming a Cram\'er type bound on prime gaps), namely long blocks of zeroes, or very long shallow $\{0,d\}$-valued blocks for some $d \geq 2$.
\end{abstract}

\maketitle

%%%%%%%%%%%%%%%%%%%%%%%%%%%%%%%%%%%%%%%%%%%%%%%%%

\section{Introduction}

\subsection{Gilbreath arrays}

Given a length $n$ sequence $a_1, a_2, a_3, \dots, a_n$ of real numbers, define the \emph{absolute difference} of this sequence to be the length $n-1$ sequence
$$ |a_1-a_2|, |a_2-a_3|, \dots, |a_{n-1}-a_n|.$$
We place this absolute difference sequence beneath the original sequence in an interspersed fashion:
$$
\begin{array}{cccccccccc}
  a_1 & & a_2 & & a_3 & \dots & a_{n-1} & & a_n  \\
  & |a_1-a_2| & & |a_2-a_3| & &  \dots & & |a_{n-1}-a_n| &\\
\end{array}
$$
Iterating this procedure, we see that every length $n$ sequence of real numbers generates a triangular array 
$$
\begin{array}{ccccccccccc}
  a_1 & & a_2 & & a_3 & & a_4 & & \dots & & a_n \\
  & a_{(1,1)} & & a_{(1,2)} & & a_{(1,3)} & & \dots & & a_{(1,n-1)} & \\
  & & a_{(2,1)} & & a_{(2,2)} & & \dots & & a_{(2,n-2)} & \\
  & & & a_{(3,1)} & & \dots & & a_{(3,n-3)} & & \\
  & & & & \ddots & & & & & \\
  & & & & & a_{(n-1,1)} & & & &
\end{array}
$$
of real numbers of sidelength $n$, where $a_{(i,j)}$ for $0 \leq i \leq n-1$ and $1 \leq j \leq n-i$ is defined by the absolute difference recursion
\begin{equation}\label{ai-recur}
a_{(i+1,j)} = |a_{(i,j)}-a_{(i,j+1)}|
\end{equation}
with initial condition
\begin{equation}\label{ai-initial}
  a_{(0,j)} = a_j
\end{equation}
for $1 \leq j \leq n$. If the initial sequence consisted of non-negative integers, then clearly the resulting array will also.  We refer to this array as the \emph{Gilbreath array} generated by this sequence, which we call the \emph{initial data} for the array.  For instance, if we take the initial data to be the first six primes $2,3,5,7,11,13$, the resulting Gilbreath array is
$$
\begin{array}{ccccccccccc}
  2 & & 3 & & 5 & & 7 & & 11 & & 13 \\
  & 1 & & 2 & & 2 & & 4 & & 2 & \\
  & & 1 & & 0 & & 2 & & 2 & \\
  & & & 1 & & 2 & & 0 & & \\
  & & & & 1 & & 2 & & & \\
  & & & & & 1 & & & &
\end{array}.
$$
One can of course start with an infinite sequence of real numbers as the initial data, which would then create an infinite Gilbreath array in the obvious fashion.

It was conjectured by Gilbreath in 1958, and independently by Proth in 1878 \cite{proth}, that the left diagonal of the Gilbreath array generated by the primes $p_1, p_2, \dots$ consists entirely of ones after the first row.  This conjecture remains open, but has been verified for the first $3.34 \times 10^{11}$ rows \cite{odlyzko}.

By removing the left diagonal and top row, dividing by two, and then subtracting one from the new top row, the conjecture is equivalent to the assertion that the left diagonal of the Gilbreath array generated by the normalized gaps $\frac{p_3-p_2}{2}-1, \frac{p_{4}-p_{3}}{2}-1, \dots$ takes values in $\{0,1\}$.  For instance, the first nine normalized gaps are $0, 0, 1, 0, 1, 0, 1, 2, 0$  (\url{https://oeis.org/A100820}), and the corresponding Gilbreath array is
$$
\begin{array}{ccccccccccccccccc}
  0 & & 0 & & 1 & & 0 & & 1 & & 0 && 1 && 2 && 0\\
  & 0 & & 1 & & 1 & & 1 & & 1 && 1 && 1 && 2\\
  & & 1 & & 0 & & 0 & & 0 &&  0  && 0 && 1\\
  & & & 1 & & 0 & & 0 & &  0 && 0 && 1\\
  & & & & 1 & & 0 & & 0 && 0 && 1\\
  & & & & & 1 & & 0 & & 0 && 1\\
  & & & & & & 1 & & 0 && 1\\
  & & & & & & & 1 && 1 \\
  & & & & & & & & 0
\end{array}.
$$
It is clear that if one row of a Gilbreath array takes values between $0$ and $D$ for some $D$, then the same is true for all subsequent rows.  Thus, for non-negative Gilbreath arrays, the maximum size attained at a row does not increase as one goes deeper into the array.  In the case of the array generated by the first $N$ normalized prime gaps for some $N>1$, we have the Cram\'er--Granville conjecture \cite{cramer} asserting an upper bound\footnote{See Section \ref{notation} for our conventions on asymptotic notation.}
\begin{equation}\label{cramer} 
  p_{n+1} - p_n \ll \log^2(2+n)
\end{equation}
on prime gaps, so the top row of this array conjecturally has maximum size $O(\log^2 N)$.  Since there are $N$ rows, it seems quite plausible for $N$ large enough that the maximum size will drop enough times as one goes deeper into the array that one will end up with either $0$ or $1$ at the bottom vertex of the array, which would suggest that the Gilbreath conjecture is true with at most finitely many exceptions.  Unfortunately, such bounds are not strong enough by themselves to resolve this conjecture, even asymptotically; see \cite{eppstein}.  As we shall discuss in more detail later, there are two scenarios in particular that could in principle derail the above heuristic analysis.  The first is that a long block of zeroes might appear at some point in the array, preventing the entries next to that block from decreasing for some time, as the following example demonstrates:
$$
\begin{array}{ccccccccccccccccc}
  2 & & 0 & & 0 & & 0 & & 0 & & 0 && 0 && 0 && 3\\
  & 2 & & 0 & & 0 & & 0 & & 0 && 0 && 0 && 3\\
  & & 2 & & 0 & & 0 & & 0 &&  0  && 0 && 3\\
  & & & 2 & & 0 & & 0 & &  0 && 0 && 3\\
  & & & & 2 & & 0 & & 0 && 0 && 3\\
  & & & & & 2 & & 0 & & 0 && 3\\
  & & & & & & 2 & & 0 && 3\\
  & & & & & & & 2 && 3 \\
  & & & & & & & & 1
\end{array}.
$$
Another bad scenario occurs when a very long block appears early in the array which only attains two values $0,d$ for some $d \geq 2$, in which case the values propagate within the set $\{0,d\}$ for a significant period of time without any decrease in magnitude, as illustrated for instance by the following example:
$$
\begin{array}{ccccccccccccccccc}
  3 & & 0 & & 3 & & 3 & & 0 & & 3 && 0 && 0 && 0\\
  & 3 & & 3 & & 0 & & 3 & & 3 && 3 && 0 && 0\\
  & & 0 & & 3 & & 3 & & 0 &&  0  && 3 && 0\\
  & & & 3 & & 0 & & 3 & &  0 && 3 && 3\\
  & & & & 3 & & 3 & & 3 && 3 && 0\\
  & & & & & 0 & & 0 & & 0 && 3\\
  & & & & & & 0 & & 0 && 3\\
  & & & & & & & 0 && 3 \\
  & & & & & & & & 3
\end{array}.
$$
A noteworthy special case of the above scenario occurs when one entry is non-zero, in which case the array essentially has the pattern of a Sierpinski triangle (or of Pascal's triangle modulo $2$):
$$
\begin{array}{ccccccccccccccccc}
  0 & & 0 & & 0 & & 0 & & 3 & & 0 && 0 && 0 && 0\\
  & 0 & & 0 & & 0 & & 3 & & 3 && 0 && 0 && 0\\
  & & 0 & & 0 & & 3 & & 0 &&  3  && 0 && 0\\
  & & & 0 & & 3 & & 3 & &  3 && 3 && 0\\
  & & & & 3 & & 0 & & 0 && 0 && 3\\
  & & & & & 3 & & 0 & & 0 && 3\\
  & & & & & & 3 & & 0 && 3\\
  & & & & & & & 3 && 3 \\
  & & & & & & & & 0
\end{array}.
$$

\subsection{Gilbreath's conjecture for Cram\'er type models}
As we do not (personally) know how to rule out the above scenarios, we cannot (personally) resolve Gilbreath's conjecture. However, we can instead study probabilistic analogues of the conjecture, in which the normalized prime gaps are replaced by random variables.  In a recent paper \cite{chase}, the first author established the following result:

\begin{theorem}[Gilbreath's conjecture for a random model]\label{thm-chase} Let $f: \N \to \N$ be a function with $2 \leq f(n) \leq \frac{1}{10}\frac{\log\log n}{\log\log\log n}$ for all sufficiently large $n$.  Let $a_1, a_2, \dots$ be a sequence of independent random variables, each uniformly distributed in $\{0,\dots,f(n)-1\}$ for each $n \geq 1$.  Then almost surely the left diagonal of the Gilbreath array generated by $a_1, a_2, \dots$ takes values in $\{0,1\}$ after finitely many rows.
\end{theorem}

However, the functions $f$ considered in \Cref{thm-chase} grow too slowly to be an accurate model for the normalized prime gaps $\frac{p_{n+2}-p_{n+1}}{2}-1$, which by the prime number theorem are of order $\frac{1+o(1)}{2} \log n$ on the average.  The Cram\'er random model \cite{cramer} then suggests that such gaps should behave like independent\footnote{The refined Cram\'er--Granville random model in \cite{granville} predicts instead some local correlations between gaps after taking residues with respect to small primes, but we will not study this refinement further here.  We use $2+ \log n$ here instead of $\log n$ to ensure that the geometric parameter is always greater than $1$.} geometric random variables of parameter $(2+\log n) / 2$; this prediction is supported (at least at the ``bulk'' scale of $\log n$) by other conjectures such as the Hardy--Littlewood prime tuples conjecture \cite{gallagher}, and is also consistent with \eqref{cramer}.

The first main result of this paper is to show that the analogue of Gilbreath's conjecture for the Cram\'er random model holds almost surely:

\begin{theorem}[Gilbreath for Cram\'er random model]\label{thm-main}  Let $a_1,a_2,\dots$ be a sequence of independent geometric random variables of parameter $2/(2+\log n)$ for each $n \geq 1$; thus the $a_n$ take values in the non-negative integers with
  \begin{equation}\label{geom} \P( a_n = k ) = \left(1 - \frac{2}{2+\log n}\right)^k \frac{2}{2+\log n}
  \end{equation}
  for all $n \geq 1$ and $k \geq 0$.  Then almost surely the left diagonal $a_{(1,1)}, a_{(2,1)}, \dots$ of the Gilbreath array generated by $a_1,a_2,\dots$ takes values in $\{0,1\}$ after finitely many rows.
\end{theorem}

This gives heuristic support for the original Gilbreath conjecture, and suggests that the conjecture is consistent with other consequences of the Cram\'er random model (e.g., the Cram\'er conjecture $p_{n+1} - p_n = O((2 + \log n)^2)$), but does not constitute a rigorous proof of that conjecture.

In fact, in Section \ref{mainthm-sec} we will establish a significantly more general statement, containing both \Cref{thm-main} and \Cref{thm-chase} as special cases.  Define a \emph{$2$-separated set} to be a set $A$ of integers such that $|a-a'| \geq 2$ for all distinct $a,a' \in A$, or equivalently $A$ does not contain a pair $m,m+1$ of consecutive integers. As noted previously, if the initial data takes values in certain $2$-separated sets, such as the even numbers or the multiples of $3$, then the successive differences will not be expected to become $0,1$-valued.  For a general class of random models, it turns out that being ``trapped'' in a $2$-separated set is in some sense the \emph{only} obstruction to the (almost sure) Gilbreath conjecture:

\begin{theorem}[Gilbreath for general random models]\label{thm-general}  For every $\eps>0$ there exists $\delta>0$ such that the following statement holds.  Let $a_1,a_2,\dots$ be a sequence of independent random variables taking values in the non-negative integers, obeying the following axioms:
  \begin{itemize}
  \item[(i)] Almost surely, one has $a_n \leq \delta n$ for all but finitely many $n$.
  \item[(ii)] For all but finitely many $n$, one has $\P( a_n \in A ) \leq 1-\eps$ for every $2$-separated set $A \subset \Z$. 
  \end{itemize}
  Then almost surely the left diagonal $a_{(1,1)}, a_{(2,1)}, \dots$ of the Gilbreath array generated by $a_1,a_2,\dots$ takes values in $\{0,1\}$ after finitely many rows.
\end{theorem}

The hypothesis (i) asserts that the $a_n$ do not grow at a linear or faster rate.  We suspect that this is the threshold growth behavior, but can only prove that the theorem breaks down once the $a_n$ grow exponentially; see \Cref{exp-rem}.  The condition (ii) asserts that the random variables $a_n$ do not concentrate in $2$-separated sets.  Some hypothesis of this general type is needed, since if for instance all the $a_n$ took values in the even integers, then the entire Gilbreath array would also be constrained to even values, and the conclusion could then very well be false.

Let us now quickly see how \Cref{thm-general} implies \Cref{thm-main}.  From \eqref{geom} we have
$$ \P(a_n = k+1) = \left(1 + O\left(\frac{1}{2+\log n}\right)\right) \P(a_n = k)$$
for any $k$, hence
$$ \P(a_n \in A+1) = \left(1 + O\left(\frac{1}{2+\log n}\right)\right) \P(a_n \in A)$$
for any set $A$.  For $2$-separated $A$, it holds that $A$ and $A+1$ are disjoint, and hence
$$ \P(a_n \in A) \leq \frac{1}{2} + O\left(\frac{1}{2+\log n}\right).$$
This gives axiom (ii) for any $\eps < 1/2$.  Next, for any fixed $\delta>0$, we see from \eqref{geom} that
$$ \P(a_n > \delta n) \leq \exp( - c \delta n / (2+\log n) )$$
for some absolute constant $c>0$, and axiom (i) then follows (with plenty of room to spare) from the Borel--Cantelli lemma.  Thus \Cref{thm-main} follows from \Cref{thm-general}.  A similar argument also recovers \Cref{thm-chase}; in fact we may now take $f(n)$ growing as fast as $\delta n$ for some small absolute constant $\delta>0$ and still obtain the conclusion of that theorem.

We briefly summarize the proof method of \Cref{thm-general} as follows.  An elementary analysis of the difference equation \eqref{ai-recur} will reveal that if an entry $a_{(n-1,1)}$ of the Gilbreath array does \emph{not} take values in $\{0,1\}$, then this creates a ``tower'' of adjacent triangles that connect $(n-1,1)$ to the top row of the array, where the top row of the $i^{\mathrm{th}}$ triangle is $\{0,d\}$-valued for some $d \geq 2$.  As it turns out, the contractive properties of the recurrence \eqref{ai-recur}, combined with the axiom (ii), ensure that it is exponentially unlikely to have so many values of the array constrained to a $2$-separated set such as $\{0,d\}$.  On the other hand, the condition (i), combined with standard counting arguments, will show that the total number of possible towers grows not much faster than $\binom{n}{\delta n}$.  For $\delta$ small enough, we will then be able to conclude from the union bound and the Borel--Cantelli lemma.

\subsection{A continuous model problem}\label{continuous model}

Next, we study a continuous random model, in which the initial input $a_n$ are now standard independent exponential distributions, so that
$$ \P(a_n \geq t) = e^{-t}$$
for all $n \geq 1$ and $t \geq 0$.  The real numbers $a_n \frac{2+\log n}{2}$ are then models for the prime gaps.  The Gilbreath array associated to these $a_n$ is now a stationary process, so in particular the distribution of the $a_{(i,j)}$ depends only on the depth $i$ and not on the position $j$; however there is a non-trivial coupling between the $a_{(i,j)}$ as $j$ varies, which becomes increasingly complicated as $i$ increases.  As each $a_{(i,j)}$ is a piecewise linear function applied to exponential random variables, it has a finite expectation, thus we have
\begin{equation}\label{c-def}
   \E a_{(i,j)} = c_i
\end{equation}
for some sequence of positive real numbers $c_0, c_1, c_2, \dots$.  The Cram\'er model (or the calculation of Gallagher \cite{gallagher}) then predicts that for the Gilbreath array for the normalized primes gaps, $a_{(n,j)}$ should behave on average like $\frac{c_j}{2} \log n$ as $n \to \infty$ for any fixed depth $j$.  If one knew that the $c_i$ decayed faster than $1/\log i$, then this would provide further support for the Gilbreath conjecture, at least in some averaged sense.

The $c_i$ are in fact rational and can in principle be computed explicitly as follows.  By a standard calculation, one can express $a_1,\dots,a_{i+1}$ as $sb_1,\dots,sb_{i+1}$, where $s \coloneqq a_1+\dots+a_{i+1}$ is a certain random variable of expectation $i+1$, and $b_1,\dots,b_{i+1}$ are independent of $s$, and uniformly distributed on the simplex of tuples $(b_1,\dots,b_{i+1})$ with $b_1,\dots,b_{i+1} \geq 0$ and $b_1+\dots+b_{i+1}=1$.  By homogeneity, we may then write
$$ c_i = (i+1) \E b_{(i,1)}$$
and the latter expectation is an integral of a piecewise linear function on the simplex with rational coefficients and therefore expressible as a rational number.  For instance,
$$ c_0 = 1,$$
$$ c_1 = 2 \int_0^1 |b_1 - (1-b_1)|\ db_1 = 1,$$
and
\begin{align*}
c_2 = 3 \times 2! \times \int_0^1 \int_0^{1-b_1} ||b_1-b_2|-|b_2-(1-b_1-b_2)||\ db_2 db_1 = \frac{7}{9} = 0.777\dots.
\end{align*}
With more effort, one can compute that $c_3 = \frac{227}{288} \approx 0.7882$; in particular, we see the unexpected phenomenon that $c_i$ is not monotone in $i$.  This irregularity of $c_i$ may be related to the fact (from Lucas' theorem) that the number of odd entries in the $i^{\mathrm{th}}$ row of Pascal's triangle is two to the number of ones in the binary expansion of $i$, which does not depend in a monotone fashion on $i$.  In Figure \ref{fig:montecarlo} we provide a Monte Carlo simulation of the subsequent values of $c_n$.

Intuitively, one might expect the $c_i$ to decrease at an exponential rate.  However, this turns out not to be the case, as we have the following lower bound:

\begin{theorem}[Lower bound]\label{thm-lower} We have $\sum_{i=0}^n c_i \geq \log(n+e)$ for any $n$.
\end{theorem}

This result turns out to be an easy consequence of the triangle inequality and Jensen's inequality; we prove it in Section \ref{clower-sec}.  
Informally, this theorem asserts that $c_i$ cannot decay faster than $1/i$, although the precise decay of $c_i$ remains somewhat mysterious.  On the other hand, from \eqref{ai-recur} and the triangle inequality we at least have the bound
$$ c_{i+1} \leq 2c_i$$
for all $i$.  

\begin{remark} Despite the irregular behavior of $c_i$, one could still tentatively conjecture that $c_i$ decays at the maximal rate $1/i$ consistent with \Cref{thm-lower}.  If this were the case, this would suggest that the upper bound of $\delta n$ in \Cref{thm-main} is close to optimal, since we now only expect Gilbreath arrays of length $n$ to decay by a factor of $n$ on average.  Unfortunately, we were not able to locate a more rigorous support for this belief; we can only show the much weaker claim that one cannot replace $\delta n$ by any quantity that is at least than $2^{n+1}$ (see \Cref{exp-rem}). We were also unable to rigorously establish the numerically and intuitively evident claims that the sequence $c_i$ converges to zero; even the weaker claim that it is bounded we cannot prove.
\end{remark}

\begin{figure}
    \centering
    \includegraphics[width=0.5\linewidth]{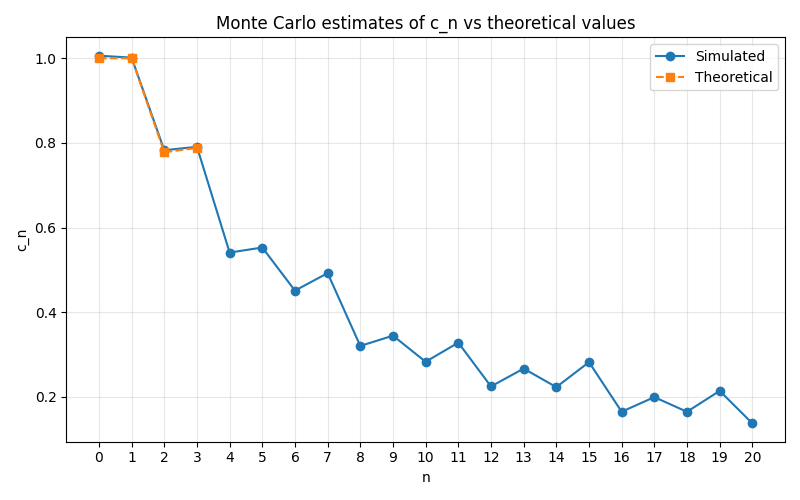}
    \caption{Approximate values of $c_n$ for $0 \leq n \leq 20$, obtained by Monte Carlo simulation from $10^6$ samples, compared against the theoretical values of $n$ for $0 \leq n \leq 3$.}
    \label{fig:montecarlo}
\end{figure}

\subsection{Deterministic analysis}

We return now to the deterministic analysis of Gilbreath arrays, in which no randomness on the initial data is assumed.  The situation is now significantly more difficult.  Nevertheless, we are able to provide an ``inverse theorem'' that makes precise the intuition that the two scenarios identified previously in the introduction, namely long $0$-valued blocks, or very long and shallow $\{0,d\}$-valued blocks for some $d \geq 2$, are in some sense the ``only'' obstructions to Gilbreath arrays decaying to zero if the initial data is not too large.  More precisely, we have

\begin{theorem}[Deterministic criterion]\label{thm-deterministic}  Let $M, L \geq 1$ and let $1 \leq N' \leq N$ be integers, and let
\begin{equation}\label{rtag} 
1 < R_0 < R_1 < R_2 < \dots < R_M < \frac{N-N'}{2}
\end{equation}
be a sequence of natural numbers such that
\begin{equation}\label{rdouble}
 R_{m} \geq 4R_{m-1}
\end{equation}
for all $m=1,\dots,M$, and also
\begin{equation}\label{R0-lower} R_0 \geq 100 L\cdot 8^M
\end{equation}
for a sufficiently large absolute constant $C_1$.
Let $a_1,\dots,a_N$ be a sequence of non-negative integers, with the resulting Gilbreath array $a_{(i,j)}$ for $0 \leq i \leq N-1$ and $1 \leq j \leq N-i$.  Assume the following axioms:
\begin{itemize}
\item[(i)] (No large initial values) One has $a_n \leq 2^M$ for all $n=0,\dots,N$.
\item[(ii)] (No long block of zeroes) There does not exist integers $0 \leq i \leq N-L$ and $1 \leq j \leq N-i-L+1$ such that $a_{(i,j)} = a_{(i,j+1)} = \dots = a_{(i,j+L-1)} = 0$.
\item[(iii)] (No long shallow two-valued block)  There does not exist integers $1 \leq m \leq M$, $2^{M-m} < d \leq 2^{M-m+1}$, $0 \leq i \leq 2R_{m-1}$, $k \geq R_m-3R_{m-1}$, and
$N' \leq j \leq N - i - k$, such that $a_{(i,j)}, a_{(i,j+1)}, \dots, a_{(i,j+k-1)} \in \{0,d\}$.
\end{itemize}
Then $a_{(N-1,1)} \in \{0,1\}$.
\end{theorem}

For instance, if we take $N$ is large, $N' = \lfloor N/2\rfloor$, $L = \lfloor \log^{10} N \rfloor$, $10 \log\log N \leq M \leq \log^{1/10} N$, and
$$ R_m \coloneqq 100 L \cdot 8^{(m+1) M},$$
then this theorem implies that we would be able to guarantee $a_{(N-1,1)} \in \{0,1\}$ provided we can establish the following hypotheses (stated somewhat informally):
\begin{itemize}
    \item[(i)]  One has $a_n \ll \log^{10} N$ (say) for all $n=0,\dots,N$.
    \item[(ii)] The array does not contain a block of zeroes of length $\sim \log^{10} N$.
    \item[(iii)]  The right half of the array does not contain a $\{0,d\}$-valued block of some depth $i \ll \exp(O(M^2))$ and length $\gg 8^M i$, for some $d \geq 2$.    
\end{itemize}
For $a_n$ equal to the normalized prime gaps, hypothesis (i) follows from Cram\'er's conjecture \eqref{cramer}.  The other two hypotheses are plausible from probabilistic heuristics, but unfortunately look difficult to establish rigorously, even if one assumes strong conjectures on the primes such as the Hardy--Littlewood prime tuples conjecture.

We now discuss some of the ideas of the proof of \Cref{thm-deterministic}.  It is convenient to work in a contrapositive setting where the hypotheses (i) and (ii) hold, but the conclusion $a_{(N-1,1)} \in \{0,1\}$ fails, in which case the task is now to produce a long shallow two-valued block.  A lengthy but elementary analysis of the recurrence \eqref{ai-recur}, using the pigeonhole principle, will eventually reveal the existence of a large triangle $\nabla^{I_*}$ which attains a large value at its bottom vertex (more than half its maximum value on the entire triangle). With an additional pigeonholing argument, one can conclude that this triangle in turn contains a large subtriangle $\nabla^J$ which only attains two values $0,d$ for some large $d$.  By additional elementary arguments we will be able to show that the upper edge $J$ of this subtriangle can be enlarged to be far longer in length while still only attaining the two values $0,d$, which will give the required contrapositive conclusion.

\subsection{Acknowledgements and AI disclosure}

The third author is supported by NSF grant DMS-2347850, and is indebted to Fedor Nazarov for supplying the arguments in \cite{nazarov}. The first author thanks David Chase for support. Gemini was used to generate the code for Figure \ref{fig:montecarlo}. 

The original proof of \Cref{thm-deterministic} (with weaker bounds) relied on an argument of Nazarov \cite{nazarov} (later improved by Anders Martinsson) which demonstrated that one could not efficiently pack a downward-pointing equilateral triangle by a finite collection of upward equilateral triangles.  However, through a successive sequence of optimizations of the proofs, there is now little trace of this original argument in the current version of the paper.

\section{Lower bound}\label{clower-sec}

We begin with the proof of \Cref{thm-lower}, which is quite short.  The key inequality is

\begin{proposition}[Key inequality]
\label{sharper} Let $c_i$ be the constants defined in Subsection~\ref{continuous model}. For any $n \geq 0$, one has
$$ c_n \geq \exp\left( - \sum_{i=0}^{n-1} c_i\right).$$
\end{proposition}

The intuition here is that if $c_0,\dots,c_{n-1}$ are not too large, then the depth $0,\dots,n-1$ rows of a Gilbreath array will not be large on average, but this, together with the triangle inequality and the exponential distribution of the initial data, will force the depth $n$ row to be large.

\begin{proof}
From \eqref{c-def}, we have
$$ c_n = \E a_{(n,1)}.$$
From \eqref{ai-recur} we have
$$ a_{(k,1)} \geq \max( a_{(k-1,1)} - a_{(k-1,2)}, 0)$$
for all $k=1,\dots,n$, and hence by induction
$$ a_{(n,1)} \geq \max\left( a_1 - \sum_{i=0}^{n-1} a_{(i,2)}, 0\right).$$
If we condition the initial data $a_2,\dots,a_n$ to be fixed, then the quantity $\sum_{i=0}^{n-1} a_{(i,2)}$ becomes deterministic, while $a_1$ remains an exponential random variable.  By a routine computation, the conditional expectation of $a_{(n,1)}$ relative to this data is then
$$ \int_0^\infty \max\left( t - \sum_{i=0}^{n-1} a_{(i,2)}, 0\right) e^{-t}\ dt = \exp\left( - \sum_{i=0}^{n-1} a_{(i,2)} \right)$$
and thus by the law of total expectation
$$ c_n \geq \E \exp\left( - \sum_{i=0}^{n-1} a_{(i,2)} \right).$$
But $\sum_{i=0}^{n-1} a_{(i,2)}$ has mean $\sum_{i=0}^{n-1} c_i$, and the claim now follows from Jensen's inequality.
\end{proof}

From Proposition \ref{sharper} we see that
$$
  \exp\left( \sum_{i=0}^{n} c_i\right) \geq (1+c_n) \exp\left( \sum_{i=0}^{n-1} c_i\right) \geq \exp\left( \sum_{i=0}^{n-1} c_i\right)+1
$$
and thus by induction
$$   \exp\left( \sum_{i=0}^{n} c_i\right)  \geq n+e.$$
Theorem \ref{thm-lower} follows.

\section{Notation}\label{notation}

We now set out the notation that we will use for the rest of the paper.
We use $\N = \{1,2,\dots\}$ to denote the natural numbers, and $\Z_{\geq 0} = \{0,1,2,\dots\}$ to denote the non-negative integers.

We use standard asymptotic notation. For quantities $X$ and $Y$ (that might depend on other parameters), we write $X \ll Y$ or $X = O(Y)$ to denote that there exists $C > 0$ such that $|X| \le C Y$ (no matter the other parameters). We write $Y \gg X$ to denote $X \ll Y$. 

\subsection{The array of locations}\label{locations-sec}

We will parameterize the infinite triangular array appearing in \Cref{thm-main} by the space
$$ \nabla^\infty \coloneqq \{ (i,j): i \in \Z_{\geq 0}; j \in \N \}$$
as depicted by the following diagram:
$$
\begin{array}{ccccccccc}
  (0,1) & & (0,2) & & (0,3) & & (0,4) & & \dots \\
  & (1,1) & & (1,2) & & (1,3) & & \dots\\
  & & (2,1) & & (2,2) & & \dots \\
  & & & (3,1) & & \dots\\
  & & & & \dots
\end{array}.
$$
We refer to elements $(i,j)$ of $\nabla^\infty$ as \emph{locations}.  We make the following definitions:
\begin{itemize}
  \item The \emph{depth} $i(p)$ of a location $p = (i,j)$ is defined to be $i$.
  \item The \emph{left position} $j_l(p)$ of a location $p = (i,j)$ is defined to be $j$.
  \item The \emph{right position} $j_r(p)$ of a location $p = (i,j)$ is defined to be $j+i$.
  \item The \emph{horizonal coordinate} $x(p)$ of a location $p = (i,j)$ is defined to be the midpoint $\frac{j_l(p) + j_r(p)}{2} = j + \frac{i}{2}$ of the left and right positions (this can be a half-integer).
  \item If $p = (i,j)$ is a location and $h$ is a integer, we define $p+h$ to be the location $(i,j+h)$ if $j+h$ is a natural number, and undefined if this is not the case.  In particular, $p+h$ is always defined if $h$ is non-negative.
  \item If $p, p'$ are two locations at the same depth $i(p) = i(p')$, we define the \emph{horizontal distance} $d(p,p')$ between them to equal 
  $$|j_l(p)-j_l(p')| = |j_r(p) - j_r(p')| = |x(p)-x(p')|.$$
  \item By abuse of notation, we identify every natural number $n$ with the depth $0$ location $(0,n)$; note that this is consistent with our previous addition notation $p+h$, as well as our notion of horizontal (or $\ell^1$) distance.
\end{itemize}
See Figure \ref{fig:location}.

\begin{figure}
    \centering
    \includegraphics[width=0.5\linewidth]{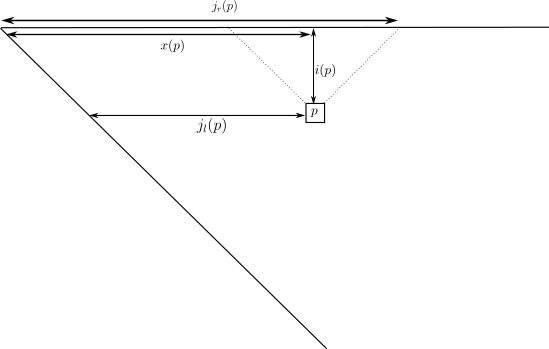}
    \caption{A location $p$ in the infinite triangle $\nabla^\infty$, together with its depth $i(p)$, left and right positions $j_l(p), j_r(p)$, and horizontal coordinate $x(p)$.}
    \label{fig:location}
\end{figure}

\begin{example} The location $(2,2)$ has depth $2$, left position $2$, right position $4$, and horizontal coordinate $3$.  The location $(2,2)+3$ is equal to $(2,5)$, and lies at a horizontal distance $3$ from $(2,2)$.
\end{example}

\begin{remark} It may be helpful for some readers to view the triangular array as a discrete $1+1$-dimensional ``spacetime'', with the depth coordinate $i(p)$ playing the role of time, and the horizontal coordinate $x(p)$ playing the role of space, though we have normalized our ``spacetime diagrams'' to have the arrow of time point downwards, rather than in the more customary upwards direction.  The equation of motion \eqref{ai-recur} corresponds to a finite speed $c$ of propagation set equal to $1/2$, and the left and right positions correspond to ``null coordinates''.
\end{remark}

We will need to consider various subsets of the triangular array $\nabla^\infty$.  We begin with the basic ``horizontal'' objects we shall need:

\begin{itemize}
  \item If $i \in \Z_{\geq 0}$, we define the \emph{$i^{\mathrm{th}}$ row} to be the set $\{ p \in \nabla^\infty: i(p) = i \}$. In particular, the top row of $\nabla^\infty$ is the zeroth row.
  \item A \emph{block} $I$ is a finite interval in a row, that is to say a set of the form
  $$I = p + [k] \coloneqq \{ p + h: h \in [k] \}$$
  for some location $p$ and natural number $k$, where we write $[k] \coloneqq \{0,1,\dots,k-1\}$.  We call $i(I) \coloneqq i(p)$, $j_l(I) \coloneqq j_l(p)$, $j_r(I) \coloneqq j_r(p) + k - 1$, and $\ell(I) \coloneqq k$  the \emph{depth}, \emph{left position}, \emph{right position}, and \emph{length} of the block respectively.  We define the \emph{central horizontal coordinate} $x(I)$ of $I$ to be the quantity $x(I) \coloneqq x(p) + \frac{k-1}{2}$; note that $I$ is symmetric around the vertical line $x = x(I)$.
  \item By abuse of notation, we identify every location $p \in \nabla^\infty$ with the length one block $p + [1]$.  Note that our notions of depth, positions,  and horizontal coordinate for locations are consistent with the corresponding notions for blocks.
  \item A \emph{subblock} of a block $I$ is a block $I'$ that is contained in $I$.
\end{itemize}

See Figure \ref{fig:block}.

\begin{figure}
    \centering
    \includegraphics[width=0.5\linewidth]{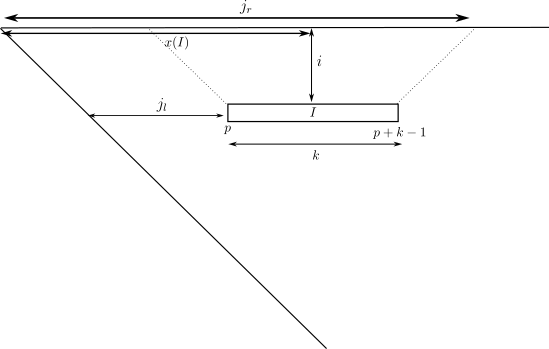}
    \caption{A block $I = p + [k]$ of depth $i$, left position $j_l$, right position $j_r$, length $k$, and central horizontal coordinate $x(I)$, with $p$ and $p+k-1$ as the two endpoints of the block.}
    \label{fig:block}
\end{figure}

\begin{example}  The block $(2,2)+[3]$ consists of the locations $(2,2)$, $(2,3)$, $(2,4)$; it lies in the second row (or equivalently, has depth $2$), has left position $2$, right position $6$, length $3$, and central horizontal coordinate $4$.  The block $(2,3)+[2]$ is a subblock of $(2,2)+[3]$ of depth $2$, left position $2$, right position $5$, length $2$, and central horizontal coordinate $3.5$.
\end{example}

We remark that blocks played a key role in the arguments in \cite{chase}, and several of our basic lemmas regarding blocks are inspired by analogous lemmas in \cite{chase}.

We introduce the following partial ordering on blocks.  If $I = (i,j) + [k]$ and $I' = (i',j') + [k']$ are blocks, we write $I \leq I'$ if $i' \geq i$, $j' = j$, and $i'+k' = i+k$; this is easily seen to be a partial ordering
(in fact it is a disjoint union of finite linear orders).  If $I < I'$, we say that $I'$ is a \emph{descendant} of $I$ (or that $I$ is an \emph{ancestor} of $I'$); in particular, if $I < I'$, then $I'$ is deeper and shorter than $I$.  Observe that a block has the same left position, right position, and central horizontal coordinate as any of its ancestors or descendants; only the depth and length change. 

A block $I$ will be maximal if and only if it has length $k=1$; otherwise it has a unique \emph{child}  $I^+ \coloneqq (i+1,j) + [k-1]$; iterating this, we see that the block $I$ has a unique \emph{ultimate descendant} $I^{+\infty} \coloneqq (i+k-1,j)$.  Similarly, a block $I$ will be minimal if and only if it has depth $i=1$, otherwise it has a unique \emph{parent} $I^- \coloneqq (i-1,j) + [k+1]$; iterating this, we see that the block $I$ has a unique \emph{ultimate ancestor} $I^{-\infty} \coloneqq j+[i+k] = (0,j) + [i+k]$. See Figure \ref{fig:order}.

\begin{example}  The block $I = (2,2)+[3]$ has descendants $(3,2) + [2]$ and $(4,2) = (4,2) + [1]$, with the former being the child $I^+$ and the latter the ultimate descendant $I^{+\infty}$.  Similarly, this block has ancestors $(1,2) + [4]$ and $2 + [5] = (0,2) + [5]$, with the former being the parent $I^-$ and the latter the ultimate ancestor $I^{-\infty}$.  In this case we have the ordering
  $$ I^{-\infty} < I^- < I < I^+ < I^{+\infty}.$$
\end{example}

\begin{figure}
    \centering
    \includegraphics[width=0.5\linewidth]{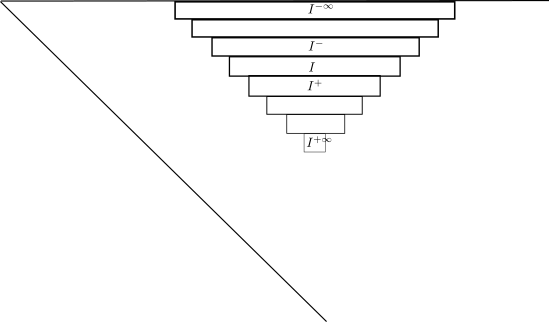}
    \caption{A block $I$ and all its comparable blocks, including its parent $I^-$, child $I^+$, ultimate ancestor $I^{-\infty}$, and ultimate descendant $I^{+\infty}$.  The union of all these blocks forms the backwards light cone $\nabla_{I^{+\infty}} = \nabla^{I^{-\infty}} = \Trap^{I^{-\infty}}_{I^{+\infty}}$. }
    \label{fig:order}
\end{figure}

Now we can define some two-dimensional regions in $\nabla^\infty$.

\begin{itemize}
  \item If $I_1, I_2$ are blocks with $I_1 \leq I_2$, we define the \emph{trapezoid} $\Trap^{I_1}_{I_2}$ to be the union of all the blocks $I$ with $I_1 \leq I \leq I_2$.  We refer to $I_1$ as the \emph{upper side} of the trapezoid, and $I_2$ as the \emph{lower side}.  
  \item If $I$ is a block, we define the \emph{triangle} $\nabla^I$ to be the trapezoid $\nabla^I \coloneqq \Trap^I_{I^{+\infty}}$ with upper side $I$, and lower side equal to the ultimate ancestor $I^{+\infty}$ of $I$.  The \emph{sidelength} $\ell(\nabla^I)$ of this triangle is defined to be the length $\ell(I)$ of $I$. We refer to the location $I^{+\infty}$ as the \emph{bottom vertex} of this triangle.
  \item If $p$ is a location, we define the \emph{backwards light cone}\footnote{One could also define a forward light cone, but we will not need this concept in this paper.} $\nabla_p$ to be the trapezoid $\nabla_p \coloneqq \Trap^{p^{-\infty}}_p = \nabla^{p^{-\infty}}$; this is a triangle of sidelength equal to $i(p)+1$.
\end{itemize}

\begin{example}  The trapezoid $\Trap^{(2,2)+[4]}_{(4,2)+[2]}$ consists of the locations $(2,2)$, $(2,3)$, $(2,4)$, $(2,5)$, $(3,2)$, $(3,3)$, $(3,4)$, $(4,2)$, and $(4,3)$.  The triangle $\nabla^{(3,2)+[3]}$ consists of the locations $(3,2)$, $(3,3)$, $(3,4)$, $(4,2)$, $(4,3)$, and $(5,2)$; it has sidelength $3$, and bottom vertex $(5,2)$.  The backwards light cone $\nabla_{(3,2)}$ consists of the locations $(1,2)$, $(1,3)$, $(1,4)$, $(2,2)$, $(2,3)$, $(3,2)$; it has sidelength $3$ and bottom vertex $(3,2)$.
\end{example}

\subsection{The difference process}

Let $\nabla_p$ be the backwards light cone for some location $p$.  We define \emph{initial data} for this light cone to be an assignment of a real number $a_j$ to each $j \in p^{-\infty}$.  Given such initial data, we may then assign a \emph{value} $a_q \in \R$ to each $q \in \nabla_p$ by solving the 
recurrence (or ``equation of motion'') \eqref{ai-recur} with $(i+1,j) \in \nabla_p \backslash p^{-\infty}$, starting from the initial condition \eqref{ai-initial} for $j \in p^{-\infty}$. creating a Gilbreath array on $\nabla_p$.
If $\Omega$ is a set of locations (such as a block, triangle, or trapezoid), we refer to the tuple $(a_p)_{p \in \Omega}$ as the \emph{values} of that set.  If $A$ is a set of reals, we say that a location $p$ is \emph{$A$-valued} if $a_p \in A$, and similarly call any set $\Omega$ of locations $A$-valued if all of its elements are $A$-valued.  If $d$ is a real, we abbreviate ``$\{d\}$-valued'' as ``$d$-valued''. We say that a set $\Omega$ \emph{attains} a value $d \in \R$ if at least one location of $\Omega$ is $d$-valued.

\begin{example}  Set $p = (3,1)$.  If we set $a_1 = 1$, $a_2 = 2$, $a_3 = 2$, and $a_4 = 0$ as our initial data, then the resulting Gilbreath array on $\nabla_p$ is
  $$
  \begin{array}{ccccccccc}
    1 & & 2 & & 2 & & 0  \\
    & 1 & & 0 & & 2 \\
    & & 1 & & 2  \\
    & & & 1 
  \end{array}
  $$
Thus, for instance, that the backwards light cone $\nabla_{(2,2)}$ is $\{0,2\}$-valued, and attains both of the values $0,2$; but the locations $1=(0,1)$, $(1,1)$, $(2,1)$, and $(3,1)$ are not $\{0,2\}$-valued.
\end{example}

Through the equations of motion \eqref{ai-recur}, the values of a block are related to those of their parents and children.  For future use, we record some specific relations of this form (cf. \cite[Lemmas 3.2, 3.3]{chase}):

\begin{lemma}[Inheritance]\label{succ-lemma}  Suppose that a backwards light cone $\nabla_p$ is given some initial data $(a_j)_{j \in p^{-\infty}}$.  Let $d \geq 0$, and let $I$ be a block in $\nabla_p$.
  \begin{itemize}
    \item[(i)] If $I$ is $\Z_{\geq 0}$-valued, then so is $I'$ for any $I' \geq I$.
    \item[(ii)] If $I$ is $\{0,1,\dots,d\}$-valued, then so is $I'$ for any $I' \geq I$.
    \item[(iii)] If $I$ is $\{0,d\}$-valued, then so is $I'$ for any $I' \geq I$.
  \end{itemize}
\end{lemma}

\begin{proof} This is immediate from the fact that the sets $\Z_{\geq 0}$, $\{0,1,\dots,d\}$ and $\{0,d\}$ are each closed under the operation $x,y \mapsto |x-y|$.
\end{proof}

\begin{lemma}[Parentage]\label{pred-lemma}  Suppose that a backwards light cone $\nabla_p$ is given some $\Z_{\geq 0}$-valued initial data $(a_j)_{j \in p^{-\infty}}$. Let $d \geq 1$, and let $I$ be a $\{0,d\}$-valued block in $\nabla_p$ of depth at least $1$ that attains the value $d$ at least once.  Then at least one of the following statements hold.
  \begin{itemize}
    \item[(i)] $I^-$ is $\{0,d\}$-valued and attains both values $0, d$.
    \item[(ii)]  $I^-$ is $\{a, a+d\}$-valued for some integer $0 < a < d$, and attains both values $a, a+d$.
    \item[(iii)] $I^-$ attains at least one integer value that is greater or equal to $2d$.
  \end{itemize}
\end{lemma}

\begin{proof} From \Cref{succ-lemma}(i), $I^-$ is $\Z_{\geq 0}$-valued.  From the contrapositive of \Cref{succ-lemma}(ii) we see that $I^-$ must attain an integer value greater than or equal to $d$.  We may assume without loss of generality that $I^-$ is $\{0,\dots,2d-1\}$-valued, since otherwise (iii) holds.  Since $I$ is $\{0,d\}$-valued, it is $d\Z$-valued, hence $I^-$ takes values in a coset $a+d\Z$ of $d\Z$ for some $0 \leq a < d$.  If $a=0$, we are in case (i), otherwise we are in case (ii).  (Note that if $I^-$ was $\{a,a+d\}$-valued but did not attain both the values $a,a+d$, then $I$ would only attain the value $0$, a contradiction.)
\end{proof}

\begin{example}  Take $d=10$. In the triangle
  $$
  \begin{array}{ccccccccc}
    17 & & 7 & & 17 & & 17  \\
    & 10 & & 10 & & 0 \\
    & & 0 & & 10 \\
    & & &  10
  \end{array}
  $$
we see that conclusion (i) of \Cref{pred-lemma} holds if we take $I$ to be the third or fourth row of this triangle, but conclusion (ii) holds if we take $I$ to be the second row.  In the trapezoid
$$
\begin{array}{ccccccccc}
  20 & & 10 & & 0 & & 0  \\
  & 10 & & 10 & & 0 \\
\end{array}
$$
we see that conclusion (iii) holds instead if we take $I$ to be the bottom row.
\end{example}

By iterating \eqref{ai-recur} we see that the value $a_{(i,j)}$ at a given location $(i,j)$ is some piecewise linear function of $a_j,\dots,a_{j+i-1}$.  This functional relationship becomes rather complicated as $i$ increases.  However, two useful facts about this relationship can be easily obtained.  The first, already observed in \cite[Lemma 3.5]{chase}, is that the \emph{parity} of the value $a_{(i,j)}$ is easy to compute:

\begin{lemma}[Parity formula]\label{parity-lemma}   Suppose that a backwards light cone $\nabla_p$ is given some $\Z$-valued initial data $(a_j)_{j \in p^{-\infty}}$.  Then for any location $(i,j)$ in $\nabla_p$\,, one has
\begin{equation}\label{aij}
 a_{(i,j)} = \sum_{k=0}^{i} \binom{i}{k} a_{j+k} \mod 2.
\end{equation}
In particular, if we fix the values of $a_j, \dots, a_{j+i-1}$ (resp. $a_{j+1},\dots,a_{j+i}$), then the parity of $a_{(i,j)}$ is determined by the parity of $a_{j+i}$ (resp. $a_j$), and vice versa.
\end{lemma}

\begin{proof}  From \eqref{ai-recur} we have
$$ a_{(i,j)} = a_{(i-1,j)} + a_{(i-1,j+1)} \mod 2$$
for any $(i,j) \in \nabla_p \backslash p^{-\infty}$.  The claim \eqref{aij} then follows by an induction on $i$ and the Pascal identity.
\end{proof}

One can of course compute the parity of $\binom{i}{k}$ using Lucas' theorem.  As it turns out, we will not use \Cref{parity-lemma} directly in our arguments; it is useful for excluding long $0$-valued blocks, or $\{0,d\}$-valued blocks for even $d$, but does not easily reduce the likelihood of long $\{0,d\}$-valued blocks for occurring for odd $d$.  For all $d$, we will use the following variant of \Cref{parity-lemma}. 

\begin{lemma}[Separation]\label{sep-lemma}  Suppose that a backwards light cone $\nabla_p$ is given some $\Z$-valued initial data $(a_j)_{j \in p^{-\infty}}$.  Let $(i,j) \in \nabla_p$ be a location, and let $D \subset \Z$ be a $2$-separated set. Fix the values of $a_j, \dots, a_{j+i-1}$ (resp. $a_{j+1},\dots,a_{j+i}$).  Then the set of possible values of $a_{j+i}$ (resp. $a_j$) that would make the location $(i,j)$ $D$-valued is also $2$-separated.
\end{lemma}

In this paper we will only apply this lemma with $D = \{0,d\}$ for various integers $d \geq 2$.

\begin{proof} We just establish the claim for $a_{j+i}$ (fixing $a_j, \dots, a_{j+i-1}$), as the claim for $a_j$ (fixing $a_{j+1},\dots,a_{j+i}$) is similar.  Consider two possible values $m,m+1$ for $a_{j+i}$.  From \eqref{ai-recur} we have
$$ a_{(1, j+i-1)} = |a_{j+i} - a_{j+i-1}|$$
so two consecutive values for $a_{j+i}$ would lead to two consecutive values for $a_{(1,j+i-1)}$, since $a_{j+i-1}$ is fixed.  A second application of \eqref{ai-recur} gives
$$ a_{(2, j+i-2)} = |a_{(1,j+i-1)} - a_{(1,j+i-2)}|$$
and again two consecutive values for $a_{(1,j+i-1)}$ would lead to two consecutive values for $a_{(2,j+i-2)}$, since $a_{(1,j+i-2)}$ is fixed.  Iterating this, we see that two consecutive values for $a_{j+i}$ would ultimately lead to two consecutive values for $a_{(i,j)}$; but the $2$-separated set $D$ does not contain two consecutive integers, and so the claim follows.
\end{proof}

\subsection{Towers}

By iterating \Cref{pred-lemma}, we can show that any location that attains a large value will form the base of a ``tower'' of triangles above that location.  To make this precise, we introduce the following definition.

\begin{definition}[Tower]\label{tower-def} Let $\nabla^{I_*}$ be a triangle and $D\ge 1$. A \emph{tower of complexity $D$ for $\nabla^{I_*}$} is an increasing sequence  
\begin{equation}\label{dk}
   2 \leq d_1 < d_2 < \dots < d_k \leq D
\end{equation}
of integers $d_1,\dots,d_k$ for some \emph{height} $k \geq 1$, together with triangles $\nabla^{I_1}, \dots,\nabla^{I_k}$ in $\nabla^{I_*}$, obeying the following axioms:
\begin{itemize}
\item[(i)] The bottom vertex $I_1^{+\infty}$ of $\nabla^{I_1}$ equals the bottom vertex $I_*^{+\infty}$ of $\nabla^{I_*}$.
\item[(ii)] For $2 \leq j \leq k$, the bottom vertex $I_j^{+\infty}$ of $\nabla^{I_j}$ lies in the parent $I_{j-1}^-$ of $I_{j-1}$; in particular, $I_{j-1}$ has depth at least one.  (This axiom is vacuously true when $k=1$.)
\item[(iii)] $I_k$ is a subblock of $I_*$.
\end{itemize}
The tower is \emph{attained} (for a given choice of initial data $a_1,\dots,a_n$) if the triangle $\nabla^{I_j}$ is $\{0,d_j\}$-valued for each $j=1,\dots,k$; see Figure \ref{fig:tower}.
\end{definition}

We then have

\begin{lemma}[Attained tower]\label{attain}  Suppose that $\nabla^{I_*}$ is a $\{0,\dots,D\}$-valued triangle with $a_{I_*^{+\infty}} > 1$.  Then at least one tower of complexity $D$ for $\nabla^{I_*}$ is attained.
\end{lemma}

\begin{proof}  Set $p_1 \coloneqq I_*^{-\infty}$ and $d_1 \coloneqq a_{p_1}$, thus $2 \leq d_1$ by hypothesis.  Applying \Cref{pred-lemma} iteratively, we can express $p_1$ as the bottom vertex $I_1^{+\infty}$ of a $\{0,d_1\}$-valued triangle $\nabla^{I_1}$ in $\nabla^{I_*}$, such that either $I_1$ is a subblock of $I_*$, or else the parent $I_1^-$ still lies in $\nabla^{I_*}$ and contains a location $p_2$ attaining some value $d_1 < d_2 \leq D$.  In the former case, we are done (with $k=1$).  Otherwise, we apply \Cref{pred-lemma} iteratively again, to express $p_2$ as the bottom vertex $I_2^{+\infty}$ of a $\{0,d_2\}$-valued triangle $\nabla^{I_2}$ in $\nabla^{I_*}$, such that either $I_2$ is a subblock of $I_*$, or else the parent $I_2^-$ still lies in $\nabla^{I_*}$ and contains a location $p_3$ attaining some value $d_2 < d_3 \leq D$.  In the former case, we are now done (with $k=2$); otherwise we continue this process in the obvious fashion.  This process must terminate after at most $D$ steps, and the claim follows.
\end{proof}

\begin{figure}
    \centering
    \includegraphics[width=0.5\linewidth]{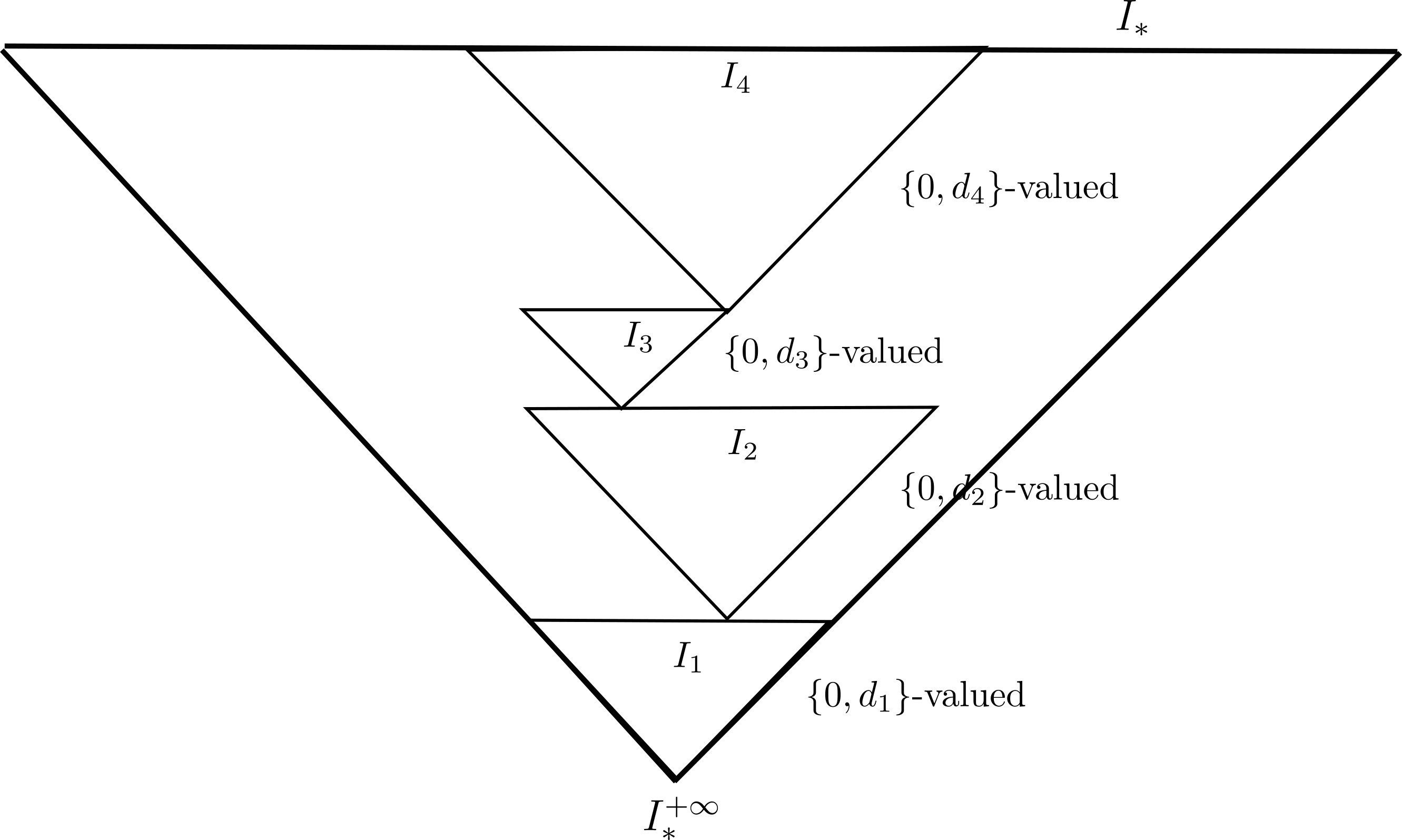}
    \caption{A schematic depiction of an attained tower of height $4$ in a triangle $\nabla^{I_*}$. }
    \label{fig:tower}
\end{figure}

Each tower has a ``large shadow'' in the following sense.

\begin{lemma}[Towers have large shadow]\label{tower-large}  Let $T$ be a tower in a triangle $\nabla^{I_*}$, and let $p$ be an element of the top row $I_*$ of the triangle.  Then there exists an element $p'$ of one of the triangles in the tower such that $p'$ either has the same left position as $p$ or the same right position as $p$, thus $j_l(p') = j_l(p)$ or $j_r(p') = j_r(p)$.  
\end{lemma}

Informally, this lemma, when combined with the pigeonhole principle, asserts that at least one of the ``left shadow'' $j_l(T)$ or the ``right shadow'' $j_r(T)$ will be large.

\begin{proof}  Suppose for contradiction that no such $p'$ exists.  Then the set $P$ must avoid both the ``southeast ray'' $\{ p' \in \nabla^{I_*} : j_l(p') = j_l(p) \}$ and the ``southwest ray'' $\{ p' \in \nabla^{I_*} : j_r(p') = j_r(p)\nabla\}$ emanating from $p$.  The union of these two rays disconnects the bottom vertex $I_*^{+\infty}$ of the triangle $\nabla^{I_*}$ from the top row $I_*$.  On the other hand, from the definition of a tower, there exists a path from $I_*^{+\infty}$ to $I_*$ within the tower that consists entirely of steps in the northwest direction $(-1,0)$ and northeast direction $(-1,1)$ only; see Figure \ref{fig:split}.  This gives the desired contradiction.
\end{proof}

\begin{figure}
    \centering
    \includegraphics[width=0.5\linewidth]{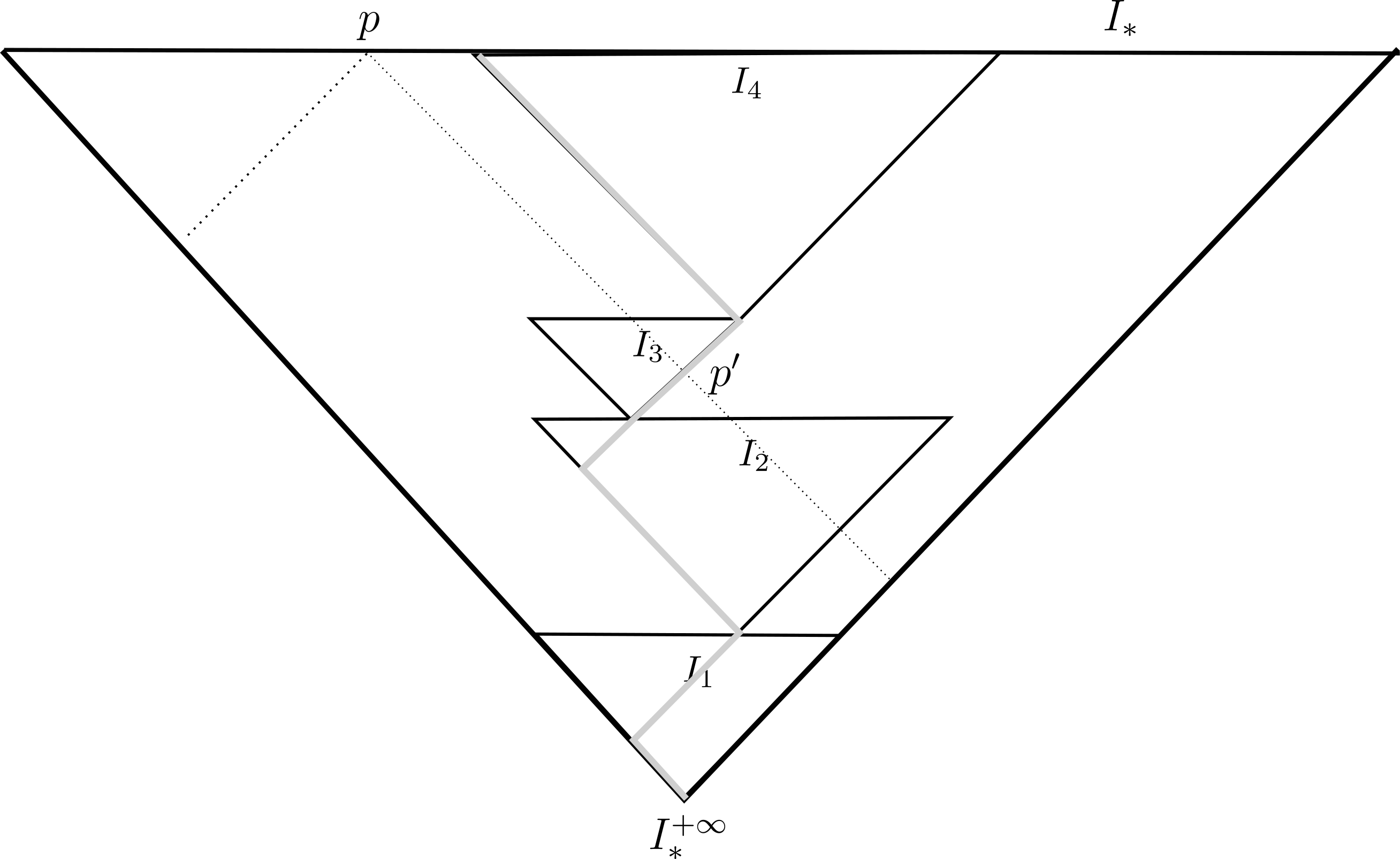}
    \caption{The union of the southeast and southwest rays from a point $p$ (depicted here as dotted lines) in $I_*$ must intersect the tower at some location $p'$, as can be seen by considering a path within that tower from the bottom vertex $I_*^{+\infty}$ to the top row $I_*$ (such as the one drawn here using shaded line segments). }
    \label{fig:split}
\end{figure}

\section{Proof of \texorpdfstring{\Cref{thm-general}}{Theorem \ref{thm-general}}}\label{mainthm-sec}

We now prove \Cref{thm-general}.  The key estimate will be the following explicit upper bound on the failure probability at a single location $(n-1,1)$.

\begin{proposition}[Bound on failure probability]\label{very-general}
Let $a_1,\dots,a_n$ be independent random variables, each taking values in the non-negative integers between $0$ and $D$. 
For each $i$, let $\rho_i \coloneqq \sup_A\P( a_i \in A)$, where $A$ ranges over $2$-separated sets $A\subset \Z$. 
Then we have
\[\P(a_{(n-1,1)}>1) \le 2^D \left( \frac{en}{D}+e\right)^{2D} \left(\prod_{i=1}^n\rho_i\right)^{1/2}.\]
\end{proposition}

\begin{remark}
    Note $\rho_i\ge 1/2$ for each $i$ (because the odd and even integers partition $\Z$ into $2$-separated sets). So this bound is non-trivial only when $D = O(n)$.
\end{remark}

Let us quickly deduce Theorem~\ref{thm-general} from this proposition.

\begin{proof}[Proof of Theorem~\ref{thm-general} assuming Proposition \ref{very-general}.]
    Fix $\eps>0$ and take $\delta=\delta(\eps)>0$ sufficiently small.
    
Let $a_1,a_2,\dots,$ be an infinite sequence of random variables satisfying the axioms (i) and (ii) of the theorem. Our task is to show that almost surely, one has $a_{(n-1,1)} \leq 1$ for all sufficiently large $n$.
    
From axiom (i), we know almost surely that for all sufficiently large $n$, one has $a_1,\dots,a_n \leq \delta n$.  Thus, if we let $B_n$ denote the bad event that $a_1,\dots,a_n \leq \delta n$ but $a_{(n-1,1)} > 1$, it suffices to show that almost surely only finitely many of the $B_n$ hold.  By the Borel--Cantelli lemma, it will suffice to establish a bound of the shape $\P(B_n) \leq 1 / n^2$ (say) for sufficiently large $n$.

Let $n$ be large.  We define the modified random variables $a'_i$ for $i=1,\dots,n$ to equal $a_i$ if $a_i \leq \delta n$, and either $0$ or $1$ (chosen with equal probability) otherwise, keeping the $a'_i$ jointly independent.  By axiom (ii), there exists an $i_0$ such that for any $i \geq i_0$ and and any $2$-separated set $A$, one has
$$ \P( a'_i \in A ) \leq \P( a_i \in A) + \P( a_i > \delta n) \leq 1 - \eps + \P( a_i > \delta n)$$
while since $A$ can only contain at most one of $0$ or $1$ we have
$$ \P( a'_i \in A ) \leq 1 - \frac{1}{2} \P( a_i > \delta n)$$
and hence on averaging we have
$$ \P( a'_i \in A ) \leq 1 - \frac{\eps}{3}.$$
If we let $a'_{(i,j)}$ be the Gilbreath array associated to the $a'_i$, then we see that when $B_n$ holds, we have $a'_i = a_i$ for all $i=1,\dots,n$, and $a'_{(n-1,1)} = a_{(n-1,1)} > 1$.  Thus we have
$$ \P(B_n) \leq \P( a'_{(n-1,1)} > 1).$$
Applying Proposition \ref{very-general}, we conclude that
$$ \P(B_n) \leq 2^{\delta n} \left( \frac{en}{\delta n}+e\right)^{2\delta n} \left( 1 - \frac{\eps}{3} \right)^{n-i_0}.$$
Since $\delta \log \frac{1}{\delta} \to 0$ as $\delta \to 0$, we conclude that for $\delta$ small enough depending on $\eps$, the right-hand side decays exponentially in $n$, and in particular is bounded by $1/n^2$ for sufficiently large $n$.  The claim follows.
\end{proof}

We are now ready to prove Proposition~\ref{very-general}.   By Lemma \ref{succ-lemma}(ii), the backwards light cone $\nabla_{(n-1,1)}$ is $\{0,\dots,D\}$-valued.
By Lemma \ref{attain}, if $a_{(n-1,1)} > 1$, then there is at least one tower $T$ of complexity $D$ for $\nabla_{(n-1,1)}$ which is attained.  By the union bound, we conclude that
$$ \P(a_{(n-1,1)}>1) \le \sum_{T\in \mathcal{T}} \P\left(T\text{ is attained}\right)$$
where $\mathcal{T}$ is the collection of all towers for  $\nabla_{(n-1,1)}$ of complexity $D$.

We can get a good bound for each summand:

\begin{lemma}[Unlikely attainment]\label{tower-prob} Each tower $T$ for $\nabla_{(n-1,1)}$ has a probability of at most $\prod_{i=1}^n\rho_i^{1/2}$ of being attained.
\end{lemma}

\begin{proof}
    Fix a tower $T$.  By \Cref{tower-large}, we see that for every $1 \leq i \leq n$, there is an element $p_i$ of the tower with either $j_l(p_i)= i$ or $j_r(p_i) = i$.  By the pigeonhole principle (and the trivial bound $\rho_i \leq 1$), we either have
\begin{equation}\label{ij}
\prod_{1 \leq i \leq n: j_l(p_i) = i} \rho_i \leq  \left(\prod_{i=1}^n\rho_i\right)^{1/2}
\end{equation}
or
$$ \prod_{1 \leq i \leq n: j_r(p_i) = i} \rho_i \leq  \left(\prod_{i=1}^n\rho_i\right)^{1/2}.$$
For sake of argument we shall assume that \eqref{ij} holds, as the other case is similar.  Let $1 \leq i_1 < \dots < i_t \leq n$ be the indices $i_s$ for which $j_l(p_{i_s}) = i_s$, thus
$$
\rho_{i_1} \dots \rho_{i_t}  \leq  \left(\prod_{i=1}^n\rho_i\right)^{1/2}.$$
In order for $T$ to be attained, we need to satisfy constraints of the form
\begin{equation}\label{apis}
 a_{p_{i_s}} \in \{0, d_{c_s}\}
 \end{equation}
for $s=1,\dots,t$, where $c_s$ is the index of the triangle in the tower $T$ that contains $p_{i_s}$.  Since $j_l(p_{i_s}) = i_s$, we see that the value of the random variable $a_{p_{i_s}}$ only depends on the values of $a_{i_s},\dots,a_n$.  Since $d_{c_s} \geq 2$, the set $\{0, d_{c_s}\}$ is $2$-separated.  From Lemma \ref{sep-lemma}, we conclude that if we condition the random variables $a_{i_s+1},\dots,a_n$ to be fixed, then in order for \eqref{apis} to hold, $a_{i_s}$ must lie in a certain $2$-separated set (depending on $a_{i_s+1},\dots,a_n$).  The conditional probability of doing so is therefore at most $\rho_{i_s}$.  Applying the law of total probability, we conclude that the probability that \eqref{apis} holds for all $1 \leq s \leq t$ simultaneously is at most
$$ \rho_{i_1} \dots \rho_{i_t},$$
giving the claim.
\end{proof}

Finally, we bound the number of towers of a given complexity and length.

\begin{lemma}[Number of towers]\label{number}
    The number of towers for $\nabla_{(n-1,1)}$ of complexity $D$ and length $k$ is at most 
    $$\binom{D-1}{k}\binom{n-1}{k-1}\binom{n+k-1}{k-1}.$$
\end{lemma}

\begin{proof}  In order to specify such a tower, it suffices to provide
\begin{itemize}
    \item A sequence of values $2\le d_1< \dots < d_k \le D$;
    \item A sequence $0 = i_k < \dots < i_1 \leq n-1$ of depths of the intervals $I_1,\dots,I_k$ defining the tower; and
    \item A sequence $1 \leq j_1(I_1) \leq \dots \leq j_l(I_k) \leq n$ of left positions for these intervals.
\end{itemize}
There is an additional constraint that the right-positions $j_r(I_1) \geq \dots \geq j_r(I_k)$ need to be non-increasing, but we will discard it.

By elementary stars-and-bars combinatorics, we can bound the number of possible choices for each of these sequences by $\binom{D-1}{k}$, $\binom{n-1}{k-1}$, and $\binom{n+k-1}{k-1}$ respectively.   The claim follows.
\end{proof}

Inserting these bounds, we conclude that
$$ \P(a_{(n-1,1)}>1) \le \left( \sum_{k=1}^{D-1} \binom{D-1}{k} \binom{n-1}{k-1}\binom{n+k-1}{k-1} \right) \prod_{i=1}^n \rho_i^{1/2}.
$$
Since 
$$ \binom{n-1}{k-1}\binom{n+k-1}{k-1}\le \binom{n+k-1}{k-1}^2\le \binom{n+D}{D}^2$$
we have from the binomial theorem that
$$ \sum_{k=1}^{D-1} \binom{D-1}{k} \binom{n-1}{k-1}\binom{n+k-2}{k-1} \leq 2^D \binom{n+D}{D}^2.$$
Using the standard bound 
$$\binom{n}{s} \leq \frac{n^s}{s!} \leq \left( \frac{en}{s} \right)^s$$
(which comes from the Taylor series bound $e^s \geq \frac{s^s}{s!}$), we obtain \Cref{very-general}, and hence \Cref{thm-general}.

\begin{remark}[Exponential upper bound]\label{exp-rem} It is possible that the upper bound $\delta n$ in Theorem \ref{thm-general} can be raised somewhat.  However, it cannot be replaced by $2^{n+1}$.  Indeed, suppose that each $a_n$ is uniformly distributed in $\{0,\dots,2^{n+1}\}$.  If one conditions the value of $a_1,\dots,a_{n-1}$ to be fixed, then $a_{(n-1,1)}$ is a piecewise linear function of $a_n$, and an easy induction shows that this function has at most $2^{n-1}$ pieces, none of which are constant.  Thus there are at most $2^n$ values of $a_n$ that would lead to $a_{(n-1,1)}$ attaining the value of $0$ or $1$, and so with conditional probability at least $1/2$, this will not be the case. So, almost surely $a_{(n-1,1)}$ will not be $\{0,1\}$-valued for infinitely many $n$.  It is surprisingly difficult to narrow the gap between the lower bound $\delta n$ and the upper bound $2^{n+1}$ significantly; this is related to our poor understanding of the sequence $c_i$.  However, we believe that the lower bound is closer to the truth.
\end{remark}

\section{Deterministic arguments}\label{deterministic-sec}

In this section we prove \Cref{thm-deterministic}. Let the notation be as in the theorem.  Applying contrapositives, it suffices to show that that axioms (i) and (ii), as well as the negated conclusion $a_{(N-1,1)} > 1$, imply the negation of (iii).  Rewriting these assumptions in the notation of Section~\ref{locations-sec}, we are now assuming
\begin{itemize}
\item[(i)] (No large initial values) The top row of the triangle $\nabla_{(N-1,1)}$ is $\{0,\dots,2^M\}$-valued.  (By \Cref{succ-lemma}(ii), this implies that the rest of $\nabla_{(N-1,1)}$ is also $\{0,\dots,2^M\}$-valued.)
\item[(ii)] The triangle $\nabla_{(N-1,1)}$ does not contain a $0$-valued block of length $L$. 
\item[($\neg$)] $(N-1,1)$ attains a value greater than $1$.
\end{itemize}
Our revised task is to establish 
\begin{itemize}
\item[($\neg$ iii)] (Long shallow two-valued block)  The subtriangle $\nabla_{(N-N', N')}$ of $\nabla_{(N-1,1)}$ contains a $\{0,d\}$-valued block $(i,j)+[k]$ of length at least $R_m - 3 R_{m-1}$ and depth at most $2R_{m-1}$ for some $1 \leq m \leq M$, $2^{M-m} < d \leq 2^{M-m+1}$.
\end{itemize}

For future reference we observe the following immediate consequence of axiom (ii):

\begin{itemize}
\item[(ii')] If a block in $\nabla_{(N-1,1)}$ is $\{0,d\}$-valued and has length at least $L$, then it must attain $d$.
\end{itemize}

\subsection{Coarse monotonicity}

We begin with a useful preliminary observation, that the Gilbreath array is ``coarsely monotone'' in the sense that a large value at a location $p$ will propagate backwards to give large values at all other locations $p'$ in the backwards light cone of $p$, up to a horizontal error of $O(2^M L)$.  We begin with a technical inductive version of this claim.

\begin{lemma}[Inductive claim]\label{syndetic-inductive}  Let $d \geq 1$, let $I$ be a $\{0,d\}$-valued block $\nabla_{(N-1,1)}$ of some depth $i(I)$ and length at most $L$ that attains $d$ at least once, and let $p'$ be a location $\nabla_{(N-1,1)}$ of depth $i(p')$ at most $i(I)$.  Then there is a location $p''$ in $\nabla_{(N-1,1)}$ at the same depth as $p'$ that attains a value of at least $d$, and whose horizontal distance $|x(p'')-x(p')|$ to $p'$ obeys the upper bound
$$ |x(p'')-x(p')| \leq \max \left( |x(p')-x(I)| - \frac{i(I)-i(p')}{2} - \frac{\ell(I)-1}{2}, 0 \right) + (2^M-d+1) \left( L + \frac{1}{2} \right).$$
\end{lemma}

\begin{proof} 
By axiom (i), we have $d \leq 2^M$.

Next, we induct on the depth $i(I)$, assuming that the claim has already been proven for smaller values of $i(I)$ (this is vacuous for $i(I) = 0$).

If $i(p') = i(I)$, then we can take $p''$ to be an element of $I$ which attains the value $d$.  As $I$ has length at most $L$, one has $|x(p'') - x(I)| \leq \frac{L}{2}$, and hence
$$ |x(p'')-x(p')| \leq |x(p')-x(I)| + \frac{L}{2},$$
which gives the claim since $d \leq 2^M$ and $\frac{\ell(I)-1}{2}+\frac{L}{2} \leq L + \frac{1}{2}$.

Thus we may now assume that $i(p') < i(I)$, so in particular $I$ has positive depth.  Consider the parent $I^-$ of $I$.  This is a block of depth $i(I)-1$, central horizontal coordinate $x(I)$, and length $\ell(I)+1$ (so at most $L+1$).  By \Cref{pred-lemma}, the parent $I^-$ is either $\{0,d\}$-valued (and attains $d$ at least once), or else attains some value $d' > d$.

First suppose that $I^-$ is $\{0,d\}$-valued and attains $d$ at least once.  If $I^-$ has length at most $L$, then we can apply the induction hypothesis and find a location $p''$ in $\nabla_{(N-1,1)}$ at the same depth as $p'$ that attains a value of at least $d$ with
\begin{align*}
|x(p'')-x(p')| &\leq \max \left( |x(p')-x(I)| - \frac{i(I)-1-i(p')}{2} - \frac{\ell(I)+1-1}{2}, 0 \right) \\
&\quad + (2^M-d+1) \left( L + \frac{1}{2} \right), 
\end{align*}
closing the induction in this case.  If instead $I^-$ has length $L+1$, then it contains two subblocks of length $\ell(I) = L$, each of which remain $\{0,d\}$-valued and of depth $i(I)-1$, and have central horizontal coordinate $x(I) \pm \frac{1}{2}$.  By axiom (ii'), both of these blocks must attain $d$.  If we let $J$ be the subblock whose central coordinate $x(J) = x(I) \pm \frac{1}{2}$ is closest to $p'$ (breaking ties arbitrarily), then
$$ |x(p')-x(J)| \leq \max \left( |x(p')-x(I)| - \frac{1}{2}, \frac{1}{2} \right)$$
and hence
\begin{align*}
    & \max \left( |x(p')-x(J)| - \frac{i(I)-1-i(p')}{2} - \frac{\ell(I)-1}{2}, 0 \right)\\
&\quad \leq \max \left( |x(p')-x(I)| - \frac{i(I)-i(p')}{2} - \frac{\ell(I)-1}{2}, 0 \right).
\end{align*}
On the other hand, by the induction hypothesis we can find a location $p''$ in $\nabla_{(N-1,1)}$ at the same depth as $p'$ that attains a value of at least $d$ with
\begin{align*} 
|x(p'')-x(p')| &\leq \max \left( |x(p')-x(J)| - \frac{i(I)-1-i(p')}{2} - \frac{\ell(I)-1}{2}, 0 \right) \\
&\quad + (2^M-d+1) \left( L + \frac{1}{2} \right).
\end{align*}
Thus we can also close the induction in this case.

Finally, suppose that $I^-$ attains a value $d'>d$ at some location $p_*$.  Applying the inductive hypothesis with $I$ replaced by the length one block $p_*$, we can find a location $p''$ in $\nabla_{(N-1,1)}$ at the same depth as $p'$ that attains a value of at least $d'$ (and hence at least $d$) with
\begin{align*}
     |x(p'')-x(p')| &\leq \max \left( |x(p')-x(p_*)| - \frac{i(I)-1-i(p')}{2}, 0 \right) \\
     &\quad + (2^M-d'+1) \left( L + \frac{1}{2} \right).
\end{align*}
Since $I^-$ has length at most $L+1$ and central horizontal coordinate $x(I)$, we have
$$ |x(p_*) - x(I)| \leq \frac{L+1}{2}.$$
Since $I$ has length at most $L$, we have
$$ \frac{\ell(I)-1}{2} \leq \frac{L-1}{2}.$$
Finally, we have $d' \geq d+1$.  We conclude that
\begin{align*} |x(p'')-x(p')| &\leq \max \left( |x(p')-x(I)| - \frac{i(I)-i(p')}{2}, 0 \right) \quad \\
&+ \frac{L-1}{2} + \frac{1}{2} + \frac{L+1}{2} + (2^M-(d+1)+1) \left( L + \frac{1}{2} \right),
\end{align*}
and we again close the induction.
\end{proof}

We now specialize this lemma to the claim that we actually need:

\begin{lemma}[Coarse monotonicity]\label{syndetic-lemma}  Let $p$ be a location in $\nabla_{(N-1,1)}$ attaining a value $d \geq 1$, and let $p'$ be a location in the backwards light cone $\nabla_p$ of $p$.  Then there is a location $p''$ in $\nabla_p$ at the same depth as $p'$ that attains a value of at least $d$, and lies within $2^M (L+\frac{1}{2})$ in horizontal distance from $p'$.
\end{lemma}
  
\begin{proof}   We apply \Cref{syndetic-inductive} with $I$ equal to the length $1$ block $p$.  Because $p'$ lies in $\nabla_p$, one has
$$ |x(p')-x(I)| \leq \frac{i(I)-i(p')}{2}.$$
Since $d \geq 1$, the claim now follows.
\end{proof}

\subsection{Locating a large triangle}

The next step is to reduce matters to analyzing a certain large subtriangle $\nabla^{I_*}$ of the triangle $\nabla_{(N-N',N')}$ which attains a large value at its bottom vertex, and where the bulk of the analysis will be performed.  More precisely, we show

\begin{proposition}[Locating a large triangle]\label{triangle-def}  There exists $1 \leq m \leq M$ and a $\{0,\dots,2^{M-m+1}\}$-valued triangle $\nabla^{I_*}$ in $\nabla_{(N-N',N)}$ whose upper edge $I_*$ is at depth $R_{m-1}$, whose lower vertex $I_*^{+\infty}$ is at depth $R_m$ (so that $\nabla^{I_*}$ is of length $R_m - R_{m-1}$), and which attains a value $d_* \in \{2^{M-m}+1,\dots,2^{M-m+1}\}$ at the bottom vertex $I_*^{+\infty}$.
\end{proposition}

\begin{proof}  For each $0 \leq m \leq M$, let $J_m$ be the depth $R_m$ row of $\nabla_{(N-N',N')}$, thus $J_0 < J_1 < \dots < J_M$ and
$$ J_m \coloneqq (R_m, N') + [N-N'-R_m].$$
See Figure \ref{fig:rows}.

\begin{figure}
    \centering
    \includegraphics[width=0.5\linewidth]{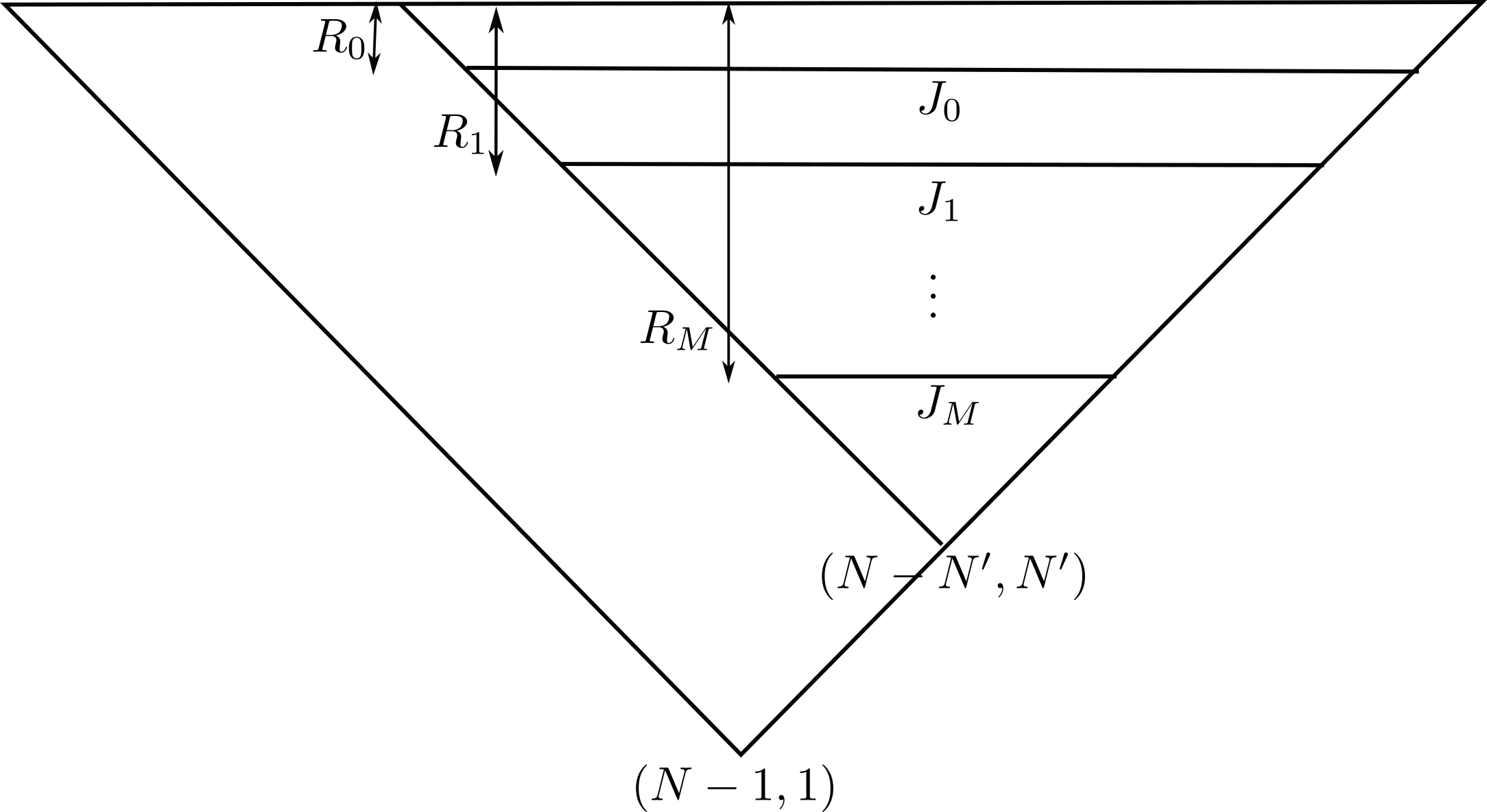}
    \caption{A schematic depiction (not to scale) of the triangle $\nabla_{(N-1,1)}$, the subtriangle $\nabla_{(N-N',N)}$, and the blocks $J_0 < \dots < J_M$ in the latter triangle at various depths $R_0 < \dots < R_M$.}
    \label{fig:rows}
\end{figure}

By \Cref{succ-lemma}(ii), $J_0$ is $\{0,\dots,2^M\}$-valued.  Meanwhile, by the negation of the original conclusion, $(n-1,1)$ attains a value greater than $1$.  Since $J_M$ lies in the backwards light cone of $(n-1,1)$ and has length
$$ \ell(J_M) = N-N'-R_M \geq R_M \geq 2R_0 \geq 2^{M+1} (L+1)$$
thanks to \eqref{rtag}, \eqref{rdouble}, and \eqref{R0-lower}, we conclude from \Cref{syndetic-lemma} that $J_M$ also attains some value greater than $1$; in particular, it is not $\{0,1\}$-valued.

To summarize, the claim ``$J_m$ is $\{0,\dots,2^{M-m}\}$-valued'' is true at $m=0$ but false at $m=M$.  Thus there must exist $1 \leq m \leq M$ such that $J_{m-1}$ is $\{0,\dots,2^{M-m+1}\}$-valued, but $J_m$ is not $\{0,\dots,2^{M-m}\}$-valued. 
 By Lemma~\ref{succ-lemma}(ii), the trapezoid $\Trap^{J_{m-1}}_{J_m}$ is $\{0,\dots,2^{M-m+1}\}$-valued, but attains some value $d_* \in \{2^{M-m}+1,\dots,2^{M-m+1}\}$ at a location $p_*$ in the lower edge $J_m$ of this trapezoid.  If we let $I_*$ be the ancestor of $p_*$ at depth $R_{m-1}$, then $I_*$ lies in $J_{m-1}$ and has length 
 $$R_m-R_{m-1}+1 \ge R_m/2$$
 thanks to \eqref{rdouble}, and hence the triangle $\nabla^{I_*}$ has sidelength $R_m - R_{m-1}+1$, is $\{0,\dots,2^{M-m+1}\}$-valued, and attains the value $d_*$ at its bottom vertex $p_*$; see Figure \ref{fig:trapezoid}.  The claim follows.
 \end{proof}

\begin{figure}
    \centering
    \includegraphics[width=0.5\linewidth]{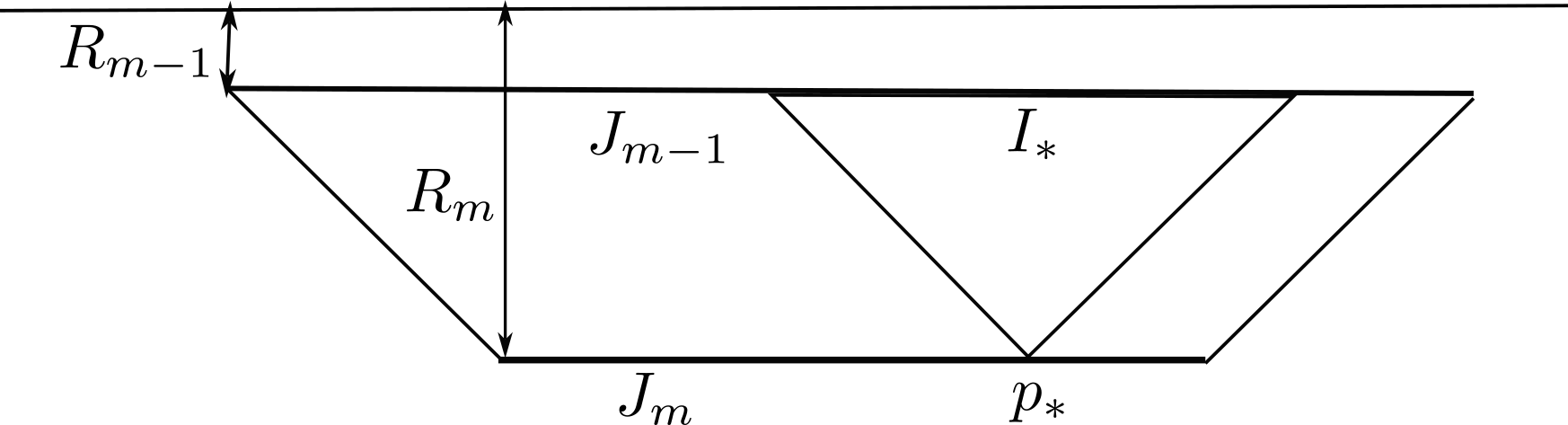}
    \caption{The triangle $\nabla^{I_*}$ inside the trapezoid $\Trap^{J_{m-1}}_{J_m}$.  This triangle attains a value $2^{M-m} < d_* \leq 2^{M-m+1}$ at its bottom vertex $p_*$, but never attains a value larger than $2^{M-m+1}$.}
    \label{fig:trapezoid} 
\end{figure}

\subsection{Good blocks}

Our strategy will be to approximately partition the triangle $\nabla^{I_*}$ produced by \Cref{triangle-def} by a certain ``good'' type of block, which we define as follows.

\begin{definition}[Good blocks]
We say that a block $I = p + [k]$ in $\nabla^{I_*}$ is \emph{good} if there is $d \in \{d_*,\dots,2^{M-m+1}\}$ such that $I$ is $\{0,d\}$-valued, attains the value $d$, and cannot be extended horizontally in either direction without losing at least one of these properties.  That is to say, the locations $p-1$, $p+k$ that bound $I$ to the immediate left and right either lie outside of $\nabla^{I_*}$ or attains a value that lies outside of $\{0,d\}$ (and hence outside of $d\Z$, since $d \geq d_* > 2^{M-m}$, and $\nabla^{I_*}$ is $\{0,\dots,2^{M-m+1}\}$-valued).  We call $d$ the \emph{non-zero value} of the good block; it is clearly uniquely determined by the block.  

\end{definition}

Clearly, if a location $p$ in $\nabla^{I_*}$ attains a value $d \in \{d_*,\dots,2^{M-m+1}\}$, then it lies in a unique good block (of non-zero value $d$), formed by extending $p$ maximally to the left and right until one either reaches the boundary of $\nabla^{I_*}$, or attains a value other than $0$ or $d$.  We record some further basic properties of $d$-good blocks:

\begin{lemma}[Basic properties of good blocks]\label{lem-good}\ 
\begin{itemize}
    \item[(i)] (Coarse disjointness) Any two distinct good blocks $I,I'$, if they intersect at all, intersect in a block of length less than $L$.
    \item[(ii)] (Parentage) If the parent $I^-$ of a good block $I$ of non-zero value $d$ is not a good block of non-zero value $d$ but still lies in $\nabla^{I_*}$, then $I^-$ attains a value $d'$ with $d < d' < 2d_*$.  Furthermore, every subblock of $I^-$ of length at least $L+1$ attains $d'$.
    \item[(iii)] (Inheritance) Every descendant $J$ of a good block $I$ of non-zero value $d$, is either a good block of non-zero value $d$, or has length less than $L$ (or both).
    \item[(iv)] (Coarse covering) Every element of $\nabla^{I_*}$ lies within $2^M (L+\frac{1}{2})$ in horizontal distance of a good block at the same depth.
  \end{itemize}
\end{lemma}

\begin{remark} Properties (i), (iv) of this lemma can be interpreted as an assertion that the good blocks form a ``coarse partition'' of $\nabla^{I_*}$.  Property (iii) implies that the good blocks (essentially) organize into (essentially) disjoint triangles, where the non-zero value of the good blocks remain constant within the triangle; and property (ii) ensures that the top of each such triangle, if not a subblock of $I_*$, launches several additional such triangles of higher non-zero value (cf., \Cref{attain}).  Our original arguments relied heavily on this triangular structure (in the spirit of \cite{collatz}); but through successive optimizations of the argument, the role of these triangles has diminished significantly, and we now argue via a more direct analysis of the good blocks.
\end{remark}

\begin{proof}  We first prove (i).  Suppose for contradiction that $I \cap I'$ is a block of length at least $L$. The block $I \cap I'$ is both $\{0,d\}$-valued and $\{0,d'\}$-valued.  By axiom (ii'), it must attain both $d$ and $d'$.  Thus $d=d'$, and $I \cap I'$ attains the value $d$.  Since a $d$-valued location is contained in a unique good block of non-zero value $d$, we obtain $I=I'$, a contradiction.

For (ii), observe that $I^-$ cannot attain any value greater than or equal to $2d \geq 2d_* > 2^{M-m+1}$.  Applying \Cref{pred-lemma}, we conclude that either $I^-$ is $\{0,d\}$-valued and attains $d$, or that $I^-$ is $\{a,a+d\}$-valued and attains $a+d$ for some $0 < a < d$.  In the former case, $I^-$ must be a good block, because if it could be extended horizontally in either direction then by \Cref{succ-lemma}(iii) the good block $I$ could also be extended horizontally, a contradiction.  So suppose we are in the latter case. Setting $d' \coloneqq a+d$, we must have $d' \leq 2^{M-m+1} < 2d_*$ since $\nabla^{I_*}$ is $\{0,\dots,2^{M-m+1}\}$-valued.  If $I^-$ contains a block $\tilde I$ of length at least $L+1$ that does not attain $d'$, then it only attains the value $a$, hence its child $\tilde I^+$ is a $\{0\}$-valued block of length at least $L$, contradicting axiom (ii).  This completes the proof of (ii).

For (iii), we may assume by induction that $J$ is the immediate descendant $I^+$ of $I$.  Without loss of generality, we may assume that $I^+$ has length at least $L$, which by axiom (ii) implies that $I^+$ cannot only attain the value $0$. By \Cref{succ-lemma}(iii), $I^+$ is $\{0,d\}$-valued, and hence attains $d$. If $I^+$ is a good block, we are done.  Otherwise we can extend $I^+$ horizontally to either the left or right while remaining $\{0,d\}$-valued.  By \Cref{pred-lemma}, this means that $I$ (which was already $\{0,d\}$-valued) could also be similarly extended horizontally while remaining $\{0,d\}$-valued (recaling that $2d \geq d_* > 2^{M-m+1}$ cannot be attained within $\nabla^{I_*}$), a contradiction, giving the claim.

Finally, the claim (iv) follows from \Cref{syndetic-lemma} and the fact that every $\{d_*, \dots,2d_*-1\}$-valued location lies in a good block.
\end{proof}

If $I$ is a good block of non-zero value $d_I$, then the triangle $\nabla^I$ is $\{0,d_I\}$-valued thanks to \Cref{succ-lemma}(iii).  It turns out that one can also gain some control on the values of the pattern to the left and right of the triangle as well, and in particular to show some strict increase in values (in a coarse sense) as one moves upwards in either the left neighborhood or right neighborhood of the triangle.  The key claim is the following variant of \Cref{syndetic-lemma}.

\begin{lemma}[Strict coarse upwards monotonicity]\label{upleft}  Let $I$ be a good block at some depth $i(I)$, length $\ell(I)$, left position $j_l(I)$, and right position $j_r(I)$, and let $p$ be a location in $\nabla^{I_*}$ that attains some value $d \geq d_*$.
\begin{itemize}
    \item (Monotonicity to the left of $\nabla^I$)  Suppose that the left gap $g \coloneqq j_l(I) - j(p)$ is positive, and that $i(I) + g + L \leq i(p) \leq i(I) + \ell(I) - L$.  Then there exists a location $p'$ with $j_l(p) \leq j_l(p') \leq j_l + L$ and $i(p) - g - L \leq i(p') < i(p)$, which attains a value $d' > d$.
    \item (Monotonicity to the right of $\nabla^I$)  Suppose that the right gap $g \coloneqq j_r(p) - j_r(I)$ is positive, and that $i(I) + g + L \leq i(p) \leq i(I) + \ell(I) - L$.  Then there exists a location $p'$ with $j_r(p) - L \leq j_r(p') \leq j_r(p)$ and $i(p) - g - L \leq i(p) < i(p)$, which attains a value $d' > d$.
\end{itemize}
\end{lemma}

\begin{proof}  We just prove the first claim, as the proof of the second is a reflection of the proof of the first.  By extending horizontally, we may place the location $p$ in a good block $J$ with non-zero value $d$.  Now we consider the successive ancestors $J^-, J^{-2}, \dots$ of $J$, continuing until the $(g+L)^\mathrm{th}$ ancestor $J^{-(g+L)}$.  This block has depth $i(p)-g-L$, and meets the triangle $\nabla^I$ in the length $L$ block $(i(p)-g-L, j_l(I)) + [L]$.  By \Cref{lem-good}(iii), the row of $\nabla^I$ at this depth $i(p)-g-L$ is good; and this row is distinct from $J^{-(g+L)}$, since otherwise $p$ would lie in $\nabla^I$, contradicting the positivity of the gap $g$.  Thus by \Cref{lem-good}(i), this ancestor $J^{-(g+L)}$ is not good.

Since $J$ is good and $J^{-(g+L)}$ is not good, we can apply \Cref{lem-good}(ii) repeatedly to locate a good block $J^{-(g+L)} < \tilde J \leq J$ of non-zero value $d$ whose ancestor $\tilde J^-$ is not good, but instead attains a value $d'$ with $d' > d \geq d_*$, as does every subblock of $\tilde J^-$ of length at least $L+1$.  

Let $K$ denote the subblock of $\tilde J^-$ consisting of the first $L+1$ elements of $\tilde J^-$ (or all of $\tilde J^-$, this block has length less than $L+1$).  Then $K$ attains $d'$ at some location $p'$.  By construction, we have $i(p) - g - L \leq i(p') < i(p)$ and $j_l(p) \leq j_l(p') \leq j_l(p) + L$, giving the claim.
\end{proof}

This gives the following key dichotomy for the sizes of good blocks.

\begin{lemma}[Blocks are small or huge]\label{lem:size-dichotomy}
    Suppose $I$ is a large good block at depth $i(I)$, length $\ell(I)$, left position $j_l(I)$, and right position $j_r(I)$. Then one of the following holds:
\begin{itemize}
    \item (Small block) One has $\ell(I)\le 100L \cdot 4^{M}$.
    \item (Huge block) One has $j_l(I) \leq j_l(I_*) + 5 L \cdot 2^M$ and $j_r(I) \geq j_r(I_*) - 5 L \cdot 2^M$.  Thus, $I$ occupies all but at most $10 L \cdot 2^M$ elements of the depth $i(I)$ row of $\nabla^{I_*}$.
    \end{itemize}
\end{lemma}

\begin{proof} We may assume that the block is large in the sense that 
\begin{equation}\label{I-large}
\ell(I) > 100 L \cdot 4^M.  
\end{equation}
We now show the left inequality $j_l(I) \leq j_l(I_*) + 5 L \cdot 2^M$.

Suppose for contradiction that $j_l(I) < j_l(I_*) + 5 L \cdot 2^M$, thus there is a gap to the left of $\nabla^I$ in $\nabla^{I_*}$ of spacing greater than $5 L \cdot 2^M$.  Applying \Cref{syndetic-lemma} (and the hypothesis that $p_*$ attains $d_*$), we may find a location $p_0 \in \nabla^{I_*}$ within the block
$$ (i(I) + \ell(I) - L, j_l(I) - 4 L \cdot 2^M) + [2^M (2L+1)]$$
that attains some value $d_0 \geq d_*$.  By construction, the left gap $g_0 = j_l(I) - j_l(p_0)$ is greater than $2^{M+1} L$, but at most $4L \cdot 2^M$.

Applying the first part of \Cref{upleft} (as well as \eqref{I-large}), we may find a further location $p_1 \in \nabla^{I_*}$ with
$j_l(p_0) \leq j_l(p_1) \leq j_l(p_0) + L$ and $i(p_0) - 5 L \cdot 2^M \leq i(p_1) < i(p_0)$, which attains a value $d_1 > d_0$.  The gap $g_1 = j_l(I) - j_l(p_1)$ thus lies between $g_0$ and $g_0-L$ and remains positive.  Iterating this procedure $2^{M-m}$ times, we may find a sequence $p_0, p_1, p_2, \dots, p_{2^{M-m+1}}$ of locations in $\nabla^{I_*}$ that attain an increasing sequence $d_0 < d_1 < \dots < d_{2^{M-m}}$ of values.  But this contradicts the hypothesis that $\nabla^{I_*}$ is $\{0,\dots,2^{M-m+1}\}$-valued (since $d_0 \geq d_* > 2^{M-m}$).

The proof of the right inequality $j_r(I) \geq j_r(p_*) - 5 L \cdot 2^M$ proceeds similarly, using the second part of \Cref{upleft} rather than the first.
\end{proof}

\subsection{Conclusion of the argument}

We can now obtain the desired conclusion ($\neg$ iii).  By \Cref{syndetic-lemma}, we can find some point $p\in \nabla^{I_*}$ with depth $2R_{m-1}$, which attains a value $d\ge d_*$.  Let $I$ be the ancestor of $p$ at depth $R_{m-1}$, then $\nabla^I$ is a length $R_{m-1}$ triangle with bottom vertex $p$.  By Lemma \ref{attain}, there is a tower $T$ of complexity $2^{M-m+1}$ for $\nabla^I$ which is attained, with initial value $d_0 = d$.  This tower can contain at most $2^{M-m+1}$ triangles, whose total lengths add up to $R_{m-1}$. Since
$$ 2^M\cdot (100 L \cdot 4^M)\le R_0\le R_{m-1},$$
we conclude from the pigeonhole principle that one of these triangles, say $\nabla^{I'}$, has length at least $100 L \cdot 4^M$.  By construction, the top edge $I'$ of this triangle is $\{0,d'\}$-valued for some $d' \geq d \geq d_*$, and attains this value $d'$.  This edge must therefore be contained in a good block, which by
\Cref{lem:size-dichotomy} has length at least the length of the depth $i(I')$ row of $\nabla^{I_*}$, minus $10 L \cdot 2^M$.  Since this row is at most $R_{m-1}$ rows deeper than the top row $I_*$ of $\nabla^{I_*}$, this good block has length at least
$$ (R_m - R_{m-1}) - R_{m-1} - 10 L \cdot 2^M \geq R_m - 3 R_{m-1}.$$
This gives the required conclusion.

\end{document}